\RequirePackage{fix-cm}
\documentclass{amsart}       % onecolumn (second format)
\usepackage{graphicx}
\usepackage{amsmath}
\usepackage{caption}
\usepackage{comment}
\usepackage{multirow}
\usepackage{tikz}
\usepackage{amssymb}
\usepackage{ytableau}
\usepackage{mathdots}
\newtheorem{Theorem}{Theorem}[section]
\newtheorem{Definition}[Theorem]{Definition}
\newtheorem{Proposition}[Theorem]{Proposition}
\newtheorem{Corollary}[Theorem]{Corollary}
\newtheorem{Lemma}[Theorem]{Lemma}

\newtheorem{Example}[Theorem]{Example}
\newtheorem{Conjecture}[Theorem]{Conjecture}

\begin{document}

\title[Monk's Rule and Giambelli's Formula for Peterson Varieties]{Monk's Rule and Giambelli's Formula for Peterson Varieties of All Lie Types 
}

\author{Elizabeth Drellich %etc.
}


\maketitle

\begin{abstract}
A Peterson variety is a subvariety of the flag variety $G/B$ which appears in the construction of the quantum cohomology of partial flag varieties.  Each Peterson variety has a one-dimensional torus $S^1$ acting on it. We give a basis of Peterson Schubert classes for $H_{S^1}^*(Pet)$ and identify the ring generators.  In type $A$ Harada-Tymoczko gave a positive Monk formula \cite{Monk}, and Bayegan-Harada gave Giambelli's formula \cite{Giambelli} for multiplication in the cohomology ring.  This paper gives Monk's rule and Giambelli's formula for all Lie types.
%Insert your abstract here. Include keywords, PACS and mathematical
%subject classification numbers as needed.
\keywords{Peterson Variety \and  Equivariant Cohomology \and Monk's Rule \and Giambelli's Formula \and Schubert Calculus}
\end{abstract}

\section{Introduction}
\noindent
Each cohomology ring of a Grassmannian or flag variety has a basis of {\bf Schubert classes} indexed by the elements of the corresponding Weyl group. Classical Schubert calculus computes the cohomology rings of Grassmannians and flag varieties in terms of the Schubert classes.    This paper will ``do Schubert calculus" in the equivariant cohomology rings of Peterson varieties. \\
\\
The Peterson variety is a subvariety of the flag variety $G/B$ parameterized by a linear subspace $H_{Pet} \subseteq \mathfrak{g}$ and a regular nilpotent operator $N_0\in \mathfrak{g}$. We define the Peterson variety to be $$Pet=\{gB\in G/B : Ad(g^{-1})N_0 \in H_{Pet}\}.$$The objects $H_{Pet}$ and $N_0$ are fully defined in Section 2.
\\
\\
Peterson varieties were introduced by D. Peterson in the 1990s. Peterson constructed the small quantum cohomology of partial flag varieties from what are now Peterson varieties.  Kostant used Peterson varieties to describe the quantum cohomology of the flag manifold \cite{Kostant} and  Rietsch gave the totally non-negative part of type $A$ Peterson varieties \cite{Rietsch}.  Insko-Yong explicitly identified the singular locus of type $A$ Peterson varieties and intersected them with Schubert varieties \cite{Insko-Yong}. Harada-Tymoczko proved that there is a circle action $S^1$ which preserves Peterson varieties \cite[Lemma 5.1 (3)]{Poset Pinball}.  We study the equivariant cohomology of the Peterson variety with respect to this action.\\
\\
We use GKM theory as a model for studying equivariant cohomology, but Peterson varieties are not GKM spaces under the action of $S^1$.  Nonetheless we are able to build several structures for $H_{S^1}^*(Pet)$ corresponding to GKM structures of $H_T^*(G/B)$:
\begin{center}
\begin{tabular}{p{5cm}|p{6cm}}
GKM Property of $H_T^*(G/B)$  & Corresponding Property of $H_{S^1}^*(Pet)$\\ \hline \hline
Injects into a direct sum of polynomial rings: \newline $H_T^*(G/B)\hookrightarrow \bigoplus \limits_{(G/B)^T} H_T^*(pt)$ &Injects into a direct sum of polynomial rings: \newline $H_{S^1}^*(Pet)\hookrightarrow \bigoplus \limits_{(Pet)^{S^1}} H_T^*(pt)$ \newline  {\color{white}spacespacespace} see Theorem \ref{thm:spanning}\\ \hline
Fixed points $(G/B)^T$ indexed by elements of the Weyl group & Fixed points $(Pet)^{S^1}$ indexed by subsets of simple roots \newline {\color{white}spacespacespace} see Proposition \ref{prop: fixed points}\\ \hline
Basis of Schubert classes \newline indexed by elements of the Weyl group & Basis of Peterson Schubert classes \newline indexed by subsets of simple roots \newline {\color{white}spacespacespace} see Theorem \ref{thm:basis}\\
\end{tabular}
\end{center}
\vspace{5mm}
% under the action of a one-dimensional torus $S^1$ which we define in Section \ref{section: gkm theory}.  Harada-Tymoczko gave the $S^1$-fixed points of the Peterson variety explicitly~\cite{Monk}.  In this paper we expand the GKM-like properties of Peterson varieties as illustrated in Figure~\ref{fig:GKM theory}.\\
\noindent
Using work by Harada-Tymoczko~\cite{Poset Pinball} and Precup~\cite{Precup}, we construct a basis for the $S^1$-equivariant cohomology of Peterson varieties in {\em all} Lie types.  This construction gives a set of classes which we call {\bf Peterson Schubert classes}.  The name indicates that the classes are projections of Schubert classes, they do not satisfy the all the classical properties of Schubert classes. From the Peterson Schubert classes we select a collection which is not only linearly independent over $H_{S^1}^*(pt)$,  but also upper triangular when appropriately ordered. Moreover they span $H_{S^1}^*(Pet)$, just as the full set of Peterson Schubert classes does.  We choose these classes by analogy with the type $A$ work \cite{Monk}; we ask whether there is a direct geometric justification for this choice. As we will show, this particular collection has especially elegant multiplication rules. \\
\\
Classical Schubert calculus asks how to multiply Schubert classes; we ask how to multiply in the basis of Peterson Schubert classes.  We give a Monk's formula for multiplying a ring generator and a module generator, and a Giambelli's formula for expressing any Peterson Schubert class in the basis in terms of the ring generators. The type $A$ the equivariant cohomology of the Peterson variety was presented by Harada-Tymoczko who gave a basis and a Monk's rule for the equivariant cohomology ring ~\cite{Monk}.  A type $A$ Giambelli's formula was given by Bayegan-Harada ~\cite{Giambelli}.  This paper extends those results to all Lie types.

\subsection{Main results}
%{\color{red}
%Giving structure coefficients for different types of cohomology rings of a range of algebraic varieties has been a fertile field of research.  While ordinary, quantum, and quantum equivariant cohomologies are also frequently studies, this paper deals with equivariant cohomology. The equivariant cohomology of Grassmannians is well studied and there are multiple tools for calculating the structure constants within the ring \cite{Molev-Sagan '99} \cite{Knutson Tao Puzzels} \cite{Molev '08} \cite{Kreiman '09} \cite{Thomas Yong '12}. This paper presents bases, generating sets, and multiplication rules for the equivariant cohomology of the Peterson variety.}\\

% Graham showed that flag varieties have a Monk's formula with non-negative integer coefficients \cite[Theorem 3.1]{Graham}.  Giving these coefficients explicitly has been the work of many including Buch \cite{Buch}, %Bergeron-S\'anchez-Ortega-Zabrocki \cite{B S-O Z}, 
 %and Lam-Shimozono \cite{Lam-Shimozono}.  In the Peterson case, we give the following formula.}\\
 \noindent
Peterson Schubert classes are indexed by subsets of the set of simple roots.  To each subset $K$ of simple roots we associate 1) a reduced word $v_K$ in the Weyl group and 2) a Peterson Schubert class $p_{v_K}$.   Classes $p_{s_i}$ corresponding to one-element sets $K$ generate $H_{S^1}^*(Pet)$ as a ring over $H_{S^1}^*(pt)$ and the set $\{p_{v_K}:K\subseteq \Delta\}$ is a module basis over $H_{S^1}^*(pt)$. \\
\\
Monk's rule is an explicit formula for multiplying an arbitrary module generator class $p_{v_K}$  by a ring generator class $p_{s_i}$.  For the Peterson variety, Monk's formula gives a set of constants $c_{i,K}^J \in H_{S^1}^*(\text{pt})$ such that 
$$
p_{s_i}p_{v_K}=\sum \limits_{J\subseteq \Delta} c_{i,K}^J \cdot p_{v_J}.
$$
\noindent
We give Monk's rule in terms of localizations of Peterson Schubert classes at fixed points $wB\in (Pet)^{S^1}$. \\
\\ 
{\bf Theorem} {\it The Peterson Schubert classes satisfy 
$$p_{s_i} \cdot p_{v_K} = p_{s_i}(w_K) \cdot p_{v_K} + \sum \limits_{\substack{K\subseteq J \subseteq \Delta \\ |J|=|K|+1}} c_{i,K}^J \cdot p_{v_J}$$
where $w_K$ is the longest word of the parabolic subgroup $W_K\subseteq W$ and the coefficients are $c_{i,K}^J=(p_{s_i}(w_J)-p_{s_i}(w_K)) \cdot \frac{p_{v_K}(w_J)}{p_{v_J}(w_J)}$.}\\
\\
Our Monk's rule is uniform across Lie types. This formula is similar in complexity to the equivariant Monk's rule for $G/B$ and has positive, although occasionally non-integral, coefficients.  Giambelli's formula is much simpler for Peterson varieties than for $G/B$.\\
\\
{\bf Theorem} {\it Giambelli's formula for Peterson Schubert classes:
$$\begin{array}{ll}
|K|! \cdot p_{v_K} = \prod \limits_{\alpha_i \in K} p_{s_i} & {\text{   if $K$ is connected and of type $A_n, B_n, C_n, F_4,$ or $G_2$}}\\
\frac{|K|!}{2} \cdot p_{v_K} = \prod \limits_{\alpha_i \in K} p_{s_i} & {\text{   if $K$ is connected and of type $D_n$}}\\
\frac{|K|!}{3} \cdot p_{v_K} = \prod \limits_{\alpha_i \in K} p_{s_i}  &{\text{   if $K$ is connected and of type $E_n$}.}\end{array} $$ }\\
\noindent
In Theorem \ref{thm:reduced words} we state a version of Giambelli's formula that is uniform across all Lie types.\\
\\

\subsection{Proof techniques}
For the proof of Monk's rule and Giambelli's formula, we modify Billey's formula and Ikeda-Naruse's excited Young diagrams for the $S^1$-equivariant cohomology of $Pet$. Billey's formula evaluates an equivariant Schubert class at a fixed point \cite{Billey}.  The evaluation gives $\sigma_v(w)\in H_T^*(pt)$ as a polynomial in the simple roots.  The map sending each simple root to a variable $t$ extends to a ring homomorphism from $H_T^*(pt)$ to $H_{S^1}^*(pt)$.  We use this map to calculate $p_v(w)$, the evaluation of the Peterson Schubert class $p_v$ at the $S^1$-fixed point $wB\in Pet$.  \\
\\
Ikeda-Naruse use excited Young diagrams to encode and compute Billey's formula \cite{EYD}. Their work focuses on Grassmannian permutations. We modify excited Young diagrams to work with the longest word in a Weyl group of classical type. Our proofs only require evaluating $p_v(w)$ when $w=w_0$ and thus these modified excited Young diagrams are sufficient for the classical Lie types.\\
\\
For the proof of Giambelli's formula in the exceptional Lie types, we give explicit computations using Sage.  The complete results for types $F_4$ and $G_2$ are included in this work. The code for the type $E$ computations is available at {http://arxiv.org/abs/1311.2678}.

%We modify both Ikeda and Naruse's excited Young diagrams and Billey's formula  in order to evaluate Peterson Schubert classes at $S^1$-fixed points of Peterson varieties.

\subsection{Structure of the paper} Section 2 lays out the basic definitions of Peterson varieties, defines the appropriate torus action, and discusses the relevant parts of GKM theory.  In Section 3 we give a basis of Peterson Schubert classes for $H_{S^1}^*(Pet)$, and in Section 4 we prove Monk's rule.  Giambelli's formula for Peterson varieties is presented in Section 5.  Section 6 describes a modification of excited Young diagrams, a major tool in the proof of Giambelli's formula.  Lastly in Section 7 we prove Giambelli's formula for each Lie type.

\section{The $S^1$ Action on Peterson Varieties}

\begin{figure}[b]
\caption{The Dynkin diagrams show the order on the simple reflections. The same order is imposed on the corresponding simple roots throughout this paper.}
\label{fig:root systems}
\begin{center}
\scalebox{.8}{
\begin{tikzpicture}

\draw[] (0,0) circle [radius=0.08];
\node [below] at (0,0)  {$s_1$};
\draw[] (1,0) circle [radius=0.08];
\node [below] at (1,0)  {$s_2$};
\draw (0.08,0) -- (.92,0);
\draw[] (2,0) circle [radius=0.08];
\node [below] at (2,0)  {$s_3$};
\draw (1.08,0) -- (1.92,0);
\draw[] (3,0) circle [radius=0.08];
\node [below] at (3,0)  {$s_{n-2}$};
\draw [dashed] (2.08,0) -- (2.92,0);
\draw[] (4,0) circle [radius=0.08];
\node [below] at (4,0)  {$s_{n-1}$};
\draw (3.08,0) -- (3.92,0);
\draw[] (5,0) circle [radius=0.08];
\node [below] at (5,0)  {$s_{n}$};
\draw (4.08,0) -- (4.92,0);
\node at (-1,0){$A_n$};

\draw[] (0,-2) circle [radius=0.08];
\node [below] at (0,-2)  {$s_1$};
\draw[] (1,-2) circle [radius=0.08];
\node [below] at (1,-2)  {$s_2$};
\draw (0.08,-2) -- (.92,-2);
\draw[] (2,-2) circle [radius=0.08];
\node [below] at (2,-2)  {$s_3$};
\draw (1.08,-2) -- (1.92,-2);
\draw[] (3,-2) circle [radius=0.08];
\node [below] at (3,-2)  {$s_{n-2}$};
\draw [dashed] (2.08,-2) -- (2.92,-2);
\draw[] (4,-2) circle [radius=0.08];
\node [below] at (4,-2)  {$s_{n-1}$};
\draw (3.08,-2) -- (3.92,-2);
\draw[] (5,-2) circle [radius=0.08];
\node [below] at (5,-2)  {$s_{n}$};
\draw (4.05,-1.95) -- (4.95,-1.95);
\draw (4.05,-2.05) -- (4.95,-2.05);
\draw (4.45,-1.9) -- (4.55,-2) -- (4.45,-2.1);
\node at (-1,-2){$B_n$};

\draw[] (0,-4) circle [radius=0.08];
\node [below] at (0,-4)  {$s_1$};
\draw[] (1,-4) circle [radius=0.08];
\node [below] at (1,-4)  {$s_2$};
\draw (0.08,-4) -- (.92,-4);
\draw[] (2,-4) circle [radius=0.08];
\node [below] at (2,-4)  {$s_3$};
\draw (1.08,-4) -- (1.92,-4);
\draw[] (3,-4) circle [radius=0.08];
\node [below] at (3,-4)  {$s_{n-2}$};
\draw [dashed] (2.08,-4) -- (2.92,-4);
\draw[] (4,-4) circle [radius=0.08];
\node [below] at (4,-4)  {$s_{n-1}$};
\draw (3.08,-4) -- (3.92,-4);
\draw[] (5,-4) circle [radius=0.08];
\node [below] at (5,-4)  {$s_{n}$};
\draw (4.05,-3.95) -- (4.95,-3.95);
\draw (4.05,-4.05) -- (4.95,-4.05);
\draw (4.55,-3.9) -- (4.45,-4) -- (4.55,-4.1);
\node at (-1,-4){$C_n$};

\draw[] (0,-6) circle [radius=0.08];
\node [below] at (0,-6)  {$s_1$};
\draw[] (1,-6) circle [radius=0.08];
\node [below] at (1,-6)  {$s_2$};
\draw (0.08,-6) -- (.92,-6);
\draw[] (2,-6) circle [radius=0.08];
\node [below] at (2,-6)  {$s_3$};
\draw (1.08,-6) -- (1.92,-6);
\draw[] (3,-6) circle [radius=0.08];
\node [below] at (3,-6)  {$s_{n-3}$};
\draw [dashed] (2.08,-6) -- (2.92,-6);
\draw[] (4,-6) circle [radius=0.08];
\node [below] at (4,-6)  {$s_{n-2}$};
\draw (3.08,-6) -- (3.92,-6);
\draw[] (5,-6) circle [radius=0.08];
\node [below] at (5,-6)  {$s_{n-1}$};
\draw (4.08,-6) -- (4.92,-6);
\draw[] (4,-5) circle [radius=0.08];
\node [above] at (4,-5)  {$s_n$};
\draw (4,-5.92) -- (4,-5.08);
\node at (-1,-6){$D_n$};

\begin{scope}[shift={(8,-3)}]
\draw[] (0,0) circle [radius=0.08];
\node [below] at (0,0)  {$s_1$};
\draw[] (1,0) circle [radius=0.08];
\node [below] at (1,0)  {$s_2$};
\draw (0.08,0) -- (.92,0);
\draw[] (2,0) circle [radius=0.08];
\node [below] at (2,0)  {$s_3$};
\draw[] (3,0) circle [radius=0.08];
\node [below] at (3,0)  {$s_{4}$};
\draw [] (2.08,0) -- (2.92,0);
\draw (1.05,.05) -- (1.95,.05);
\draw (1.05,-.05) -- (1.95,-.05);
\draw (1.45,.1) -- (1.55,0) -- (1.45,-.1);
\node at (-1,0){$F_4$};
\end{scope}

\begin{scope}[shift={(8,-5)}]
\draw[] (0,0) circle [radius=0.08];
\node [below] at (0,0)  {$s_1$};
\draw[] (1,0) circle [radius=0.08];
\node [below] at (1,0)  {$s_2$};
\draw (0.08,0) -- (.92,0);
\draw (0.05,.05) -- (0.95,.05);
\draw (.05,-.05) -- (.95,-.05);
\draw (.55,.1) -- (.45,0) -- (.55,-.1);
\node at (-1,0){$G_2$};
\end{scope}

\begin{scope}[shift={(8,-1)}]
\draw[] (0,0) circle [radius=0.08];
\node [below] at (0,0)  {$s_1$};
\draw[] (1,0) circle [radius=0.08];
\node [below] at (1,0)  {$s_3$};
\draw (0.08,0) -- (.92,0);
\draw (1.08,0) -- (1.92,0);
\draw[] (2,0) circle [radius=0.08];
\node [below] at (2,0)  {$s_4$};
\draw[] (3,0) circle [radius=0.08];
\node [below] at (3,0)  {$s_{n-1}$};
\draw [dashed] (2.08,0) -- (2.92,0);
\draw[] (4,0) circle [radius=0.08];
\node [below] at (4,0)  {$s_{n}$};
\draw [] (3.08,0) -- (3.92,0);
\draw[] (2,1) circle [radius=0.08];
\node [above] at (2,1)  {$s_{2}$};
\draw [] (2,.08) -- (2,.92);
\node at (-1,0){$E_n$};
\end{scope}

\end{tikzpicture}}
\end{center}
\end{figure}
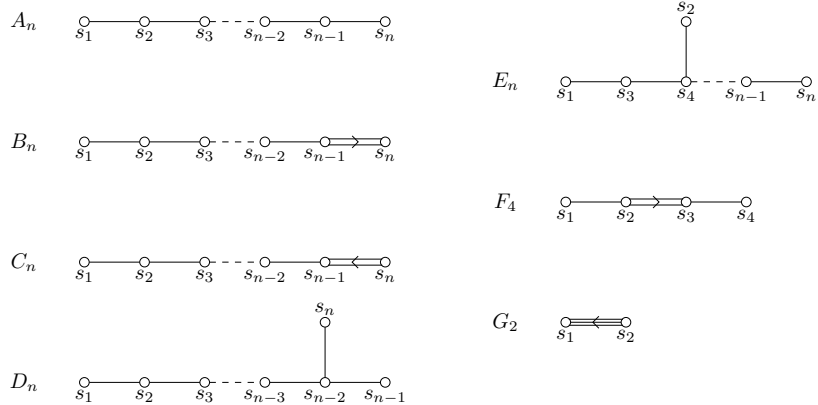

Fix a complex reductive linear algebraic group $G$, a Borel subgroup $B$, and a maximal torus $T\subseteq B \subseteq G$. This choice gives 
\begin{itemize}
\item a root system $\Phi$
\item positive roots $\Phi^+ \subset \Phi$
\item simple roots $\Delta \subset \Phi^+$
\item an associated Weyl group $W$
\item associated Lie algebras $\mathfrak{t}\subseteq \mathfrak{b} \subseteq \mathfrak{g}$
\item root spaces $\mathfrak{g_\alpha} \subset \mathfrak{g}$ for each root $\alpha \in \Phi$.
\end{itemize}
\noindent
We also choose a basis element $E_\alpha \in \mathfrak{g}_\alpha$ for each of the root spaces.  Some of our constructions rely on a specific ordering of the roots $\alpha_1, \alpha_2,\ldots, \alpha_{|\Delta |} \in \Delta$.  This ordering can be expressed in the Dynkin diagram of $\Delta$.  Figure \ref{fig:root systems} orders the simple reflections for each Lie type.\\
\\  For any Lie type the Peterson subspace in $\mathfrak{g}$ is the direct sum of $\mathfrak{b}$ and the root spaces corresponding to the negative simple roots:
$$
H_{Pet}= \mathfrak{b} \oplus \bigoplus \limits_{\alpha \in -\Delta} \mathfrak{g}_\alpha.
$$
The regular nilpotent operator $N_0 \in \mathfrak{g}$ is 
$$
N_0= \sum \limits_{\alpha \in \Delta} E_\alpha.
$$
In type $A$ the operator $N_0$ is the nilpotent with one Jordan block.
\begin{Definition}
The Peterson variety $Pet$  is a subvariety of the flag variety defined by 
$$Pet= \lbrace gB \in G/B : Ad(g^{-1})(N_0) \in H_{Pet} \rbrace.$$
\end{Definition}
\noindent Peterson varieties are a family of regular nilpotent Hessenberg variety.  They are generally irreducible and generally not smooth \cite{Insko-Yong}.  

\subsection{GKM theory and Billey's formula}
Named for Goresky, Kottwitz, and MacPherson, GKM theory expresses the $T$-equivariant cohomology of certain spaces in terms of polynomials corresponding to $T$-fixed points \cite{GKM}.  The flag variety $G/B$ with the action of a maximal torus $T\subseteq B\subseteq G$ is such a space.  \\
\\
The structure of $H_T^*(G/B)$ is encoded in the combinatorics of the Weyl group $W$. The elements of $W$ index the $T$-equivariant Schubert classes.  Each class can be thought of as a tuple of polynomials, one for each $T$-fixed point.  As the $T$-fixed points are also indexed by the Weyl group, any ordered pair $v,w \in W$  determines a polynomial $\sigma_v(w)\in \mathbb{C}[\alpha_i:\alpha_i\in \Delta]$.  This polynomial  is the Schubert class $\sigma_v$ evaluated at the fixed point $wB\in G/B$.  \\
\\
Billey gave an explicit combinatorial formula for computing the polynomial $\sigma_v(w)$~\cite{Billey}.  Fix a reduced word for $w=s_{b_1}s_{b_2} \cdots s_{b_{\ell(w)}}$ and for $j\leq \ell(w)$ let $$\mathbf{r}(\mathbf{j},w)=s_{b_1}s_{b_2} \cdots s_{b_{i-1}}(\alpha_{b_i}).$$  Then

\begin{equation}
\label{eq:Billeys}
\sigma_v(w)= \sum \limits_{\substack{\text{reduced words} \\ v=s_{b_{j_1}}s_{b_{j_2}} \cdots s_{b_{j_{\ell(v)}}} }} \left( \prod \limits_{i=1}^{\ell(v)}  \mathbf{r}(\mathbf{j_i},w) \right).
\end{equation}
\noindent
\begin{Proposition} [Billey~\cite{Billey}]
\label{prop:billey} Properties of the polynomial $\sigma_v(w)$:
\begin{enumerate}
\item \label{billey1} The polynomial $\sigma_v(w)$ is homogeneous of degree $\ell(v)$.
\item \label{billey2} If $v \not\leq w$ then $\sigma_v(w)=0$.
\item \label{billey3} If $v\leq w$ then $\sigma_v(w) \neq 0$.
\item \label{billey4} The polynomial $\sigma_v(w)$ has non-negative integer coeffiecients, i.e. $$\sigma_v(w) \in \mathbb{Z}_{\geq 0} [\alpha_i:\alpha_i\in \Delta].$$
\item \label{billey5} The polynomial $\sigma_v(w)$ does not depend on the choice of reduced word for $w$.
\end{enumerate}
\end{Proposition}
%\noindent
%When $v$ and $w$ are words of relatively short length it is simple to calculate $\sigma_v(w)$ by hand.
\begin{Example}
Let $G/B$ have Weyl group $W=A_2$ and let $w=s_1s_2s_1$ and $v=s_1$.  The word $v$ is found as a subword of $s_1s_2s_1$ in the two places $\mathbf{s_1} s_2s_1$ and $s_1s_2\mathbf{s_1}$. 
$$\sigma_v(w)=\mathbf{r}(\mathbf{1},s_1s_2s_1) + \mathbf{r}(\mathbf{3},s_1s_2s_1)= \alpha_1 + s_1s_2(\alpha_1)=\alpha_1+ \alpha_2.$$
\end{Example}

\subsection{GKM Theory and Peterson Varieties}
\label{section: gkm theory}
GKM theory applies to the full flag variety with the action of a maximal torus $T$ ~\cite{GKM}. Since the Peterson variety is a subvariety of the flag variety, it is natural to ask if it has a similar torus action.  The torus $T$ does not preserve $Pet$ but a one-dimensional subtorus $S^1 \subseteq T$ does.  In type $A$ this subtorus is
$$S^1=
\begin{bmatrix}
t & 0 & \cdots & 0 &0\\
0 & t^2 & \cdots & 0 &0\\
\vdots & \vdots & \ddots &\vdots & \vdots \\
0&0&\cdots & t^{n-1} &0\\
0&0&\cdots &0 &t^n\\
\end{bmatrix}
\subseteq
\begin{bmatrix}
t_1 & 0 & \cdots & 0 &0\\
0 & t_2 & \cdots & 0 &0\\
\vdots & \vdots & \ddots &\vdots & \vdots \\
0&0&\cdots & t_{n-1} &0\\
0&0&\cdots &0 &t_n\\
\end{bmatrix}=T.
$$
\noindent
We define the one-dimensional torus $S^1$ in general Lie type. \\
\begin{Definition}
~\cite[Lemma 5.1]{Poset Pinball}
The characters $\alpha_1, \ldots \alpha_n\in \mathfrak{t}^*$ are a maximal $\mathbb{Z}$-linearly independent set in $\mathfrak{t}^*$.  Let $\phi:T \to (\mathbb{C}^*)^n$ be the isomorphism of linear algebraic groups $t\mapsto (\alpha_1(t), \alpha_2(t),...,\alpha_n(t))$.  Define a one-dimensional torus $S^1$ by
$$S^1= \phi^{-1}(\lbrace (c,c, \ldots , c) : c \in \mathbb{C}^* \rbrace ).$$
\end{Definition}
\begin{Proposition} \cite[Lemma 5.1]{Poset Pinball}
\label{prop: fixed points}
The torus $S^1$ acts on the Peterson variety.
\end{Proposition}
\noindent Any point in $Pet$ fixed by $T$ will also be fixed by $S^1$.  In fact these are the only points in the Peterson variety fixed by $S^1$:
$$(Pet)^{S^1}=Pet \cap (G/B)^T.$$

\noindent
Harada and Tymoczko gave the $S^1$-fixed points of $Pet$ explicitly.  Let $K\subseteq \Delta$ be a subset of the simple roots. Define $W_K\subseteq W$ as the parabolic subgroup generated by $K$ and let $w_K$ be the longest element of $W_K$.  

\begin{Proposition}~\cite[Proposition 5.8]{Poset Pinball}
An element $wB\in G/B$ is an $S^1$-fixed point of $Pet$ if and only if $w=w_K$ for some set $K\subseteq \Delta$.
\end{Proposition}
\noindent
{\bf Notation} We frequently refer to the fixed point $w_KB\in Pet$ by the coset representative $w_K$.\\
\\
Although $Pet$ has a torus action and torus-fixed points indexed by Weyl group elements, it is not a GKM space.  For Peterson varieties we must build the GKM-like structures we want.

\section{Peterson Schubert classes as a basis of $H_{S^1}^*(Pet)$}
Harada-Tymoczko gave a  projection from $H_T^*(G/B)$ to $H_{S^1}^*(Pet)$ in classical Lie types \cite[Theorem 5.4]{Poset Pinball}.  In this section we extend their results to all Lie types.  For classical Lie types they gave the following commutative diagram which this section will extend to all Lie types.%Theorem \ref{thm:spanning} shows that the map is a projection and Theorem \ref{thm:basis} proves it is surjective. As a result of these two theorems, the following is a commutative diagram.
\newcommand*{\longhookrightarrow}{\ensuremath{\lhook\joinrel\relbar\joinrel\rightarrow}}
\begin{equation}
\label{eq:map diagram}
\begin{matrix}
H_T^*(G/B) & \longhookrightarrow & \bigoplus \limits_{(G/B)^T} H_T^*(pt) \\
 \downarrow && \downarrow  \pi_1 \\
H_{S^1}^*(G/B)& \longhookrightarrow & \bigoplus \limits_{(G/B)^{S^1}} H_{S^1}^*(pt)\\
 \downarrow  & & \downarrow  \pi_2\\
H_{S^1}^*(Pet ) & \longhookrightarrow  & \bigoplus \limits_{(Pet)^{S^1}} H_{S^1}^*(pt)\\
\end{matrix}
\end{equation}
 A priori $H_T^*(G/B)$ is a module over $\mathbb{C}[\alpha_i:\alpha_i\in \Delta]$.  The map $\pi_1: H_T^*(pt) \to H_{S^1}^*(pt)$  is the ring homomorphism which takes simple roots $\alpha_i \in \Delta$ to the variable $t$.  The map $\pi_2$ forgets the $T$-fixed points of $G/B$ that are not in the Peterson variety. The top two injectivities are a direct result of GKM theory \cite{GKM}. We prove the bottom injectivity in Theorem \ref{thm:spanning}.

\subsection{Peterson Schubert classes}
The image of a Schubert class $\sigma_v\in \bigoplus \limits_{(G/B)^T} H_T^*(pt)$ in $\bigoplus\limits_{(Pet)^{S^1}} H_{S^1}^*(Pet)$ is denoted $p_v$ and called a {\bf Peterson Schubert class}.   The class $p_v$ has one polynomial for each $S^1$-fixed point of $Pet$  so a Peterson Schubert class can be thought of as a $2^{|\Delta |}$-tuple of polynomials in $\mathbb{C}[t]$.    Below is an example in type $A_2$.

$$
\bordermatrix{ ~ &\sigma_{s_1} \cr
1 & 0 \cr
s_1 & \alpha_1 \cr
s_2 & 0 \cr
s_1s_2 & \alpha_1 \cr
s_2s_1 & \alpha_1+\alpha_2 \cr
s_1s_2s_1 & \alpha_1+ \alpha_2 \cr}
\begin{matrix}{\pi_1}\\
\longmapsto 
\end{matrix}
\bordermatrix{ ~ & \cr
& 0 \cr
& t \cr
& 0 \cr
& t \cr
& 2t \cr
& 2t \cr}
\begin{matrix}{\pi_2}\\
\longmapsto 
\end{matrix}
\bordermatrix{ ~ & p_{s_1} \cr
& 0 \cr
 & t \cr
& 0 \cr
& \cr
&  \cr
 & 2t \cr}.
$$
\noindent

{
\begin{Theorem}
\label{thm:spanning}
The map $H_{S^1}^*(Pet) \to \bigoplus \limits_{(Pet)^{S^1}} H_{S^1}^*(pt)$ induced by the inclusion $(Pet)^{S^1} \hookrightarrow Pet$ is an injection. 
\end{Theorem}
\noindent
Theorem \ref{thm:spanning} is a generalization of Harada-Tymoczko's Theorem 5.4 to all Lie types \cite{Poset Pinball}.
\begin{proof}
We start by showing that the ordinary cohomology of $Pet$ vanishes in odd degree. Precup proved that $Pet$ is paved by complex affines for any Lie type ~\cite[Theorem 5.4]{Precup}.  Precup also showed that the compact cohomology of the Peterson variety is only supported in even dimensions \cite[Lemma 2.7]{Precup}. Because the Peterson variety is compact, its ordinary cohomology vanishes in odd degree. \\
\\
Following Harada-Tymoczko \cite[Remark 4.11]{Poset Pinball}, the Leray-Serre spectral sequence of the Borel-equivariant cohomology of $Pet$ collapses and thus $H_{S^1}^*(Pet)$ is a free $H_{S^1}^*(pt)$-module.  Therefore the inclusion $(Pet)^{S^1}\hookrightarrow Pet$ induces an injection $$H_{S^1}^*(Pet)\hookrightarrow  \bigoplus \limits_{(Pet)^{S^1}} H_{S^1}^*(pt).$$
This concludes the proof.
%So $Pet$ is equivariantly formal~\cite{ConMath GKM Theory p174?}.  These results show that Harada-Tymoczko's Lemma 5.3 extends to all Lie types and the remainder of their proof of Theorem 5.4 is type-independent  ~\cite[Theorem 5.4]{Poset Pinball}.
  \end{proof}
\noindent
In the terminology of Harada-Tymoczko, since $(Pet)^{S^1}= Pet \cap (G/B)^T$ the pair $(Pet, S^1)$ is GKM-compatible with $(G/B,T)$. }
 
\subsection{A basis of Peterson Schubert classes}
The $S^1$-fixed points of $Pet$ are indexed by subsets $K\subseteq \Delta$ so we want to index the Peterson Schubert classes by $K\subseteq \Delta$. \\

\begin{Definition}
A subset of simple roots $K\subseteq \Delta$ is called {\bf connected} if the induced Dynkin diagram of $K$ is a connected subgraph of the Dynkin diagram of $\Delta$. 
\end{Definition}
Any subset $K\subseteq \Delta$ can be written as $K=K_1\times \cdots \times K_m$ where each $K_i$ is a maximally connected subset.  Each connected subset corresponds to its own Lie type.
\begin{Definition}
\label{def: connected part 1}
Let $K\subseteq\Delta$ be a connected subset.  We define $v_K\in W_K$ to be
$$v_K=\prod \limits_{\rm{Root}_K(i)=1}^{|K|} s_i$$ where ${\rm{Root}}_K(i)$ is the index of the corresponding root in a root system of the same Lie type as $K$, ordered as in Figure \ref{fig:root systems}.  If $K=K_1\times \cdots \times K_m$ and each $K_i$ is maximally connected then $v_K=v_{K_1}v_{K_2}\cdots v_{K_m}$.
\end{Definition}
\noindent
When $\Delta$ is not of type $D$ or $E$ this definition gives $v_K=s_{a_1} s_{a_2}\cdots s_{a_m}$ where $K=\{\alpha_{a_1},\alpha_{a_2},\ldots \alpha_{a_m}\}$ and  $a_1<a_2<\cdots a_m$.  This is the definition given in type $A$ by Harada-Tymoczko~\cite{Monk}.  Example \ref{ex: type E vk} illustrates how Definition \ref{def: connected part 1} differs from the type $A$ definition.
\begin{Example}
\label{ex: type E vk}
Let $\Delta=\{\alpha_1, \alpha_2,\alpha_3, \alpha_4, \alpha_5, \alpha_6\}$ be a the set of simple roots of a type $E_6$ root system and let $K=\Delta \setminus \{\alpha_6\}$.  The subset $K\subseteq \Delta$ represented by a marked set of vertices in the Dynkin diagram and compared to the Dynkin diagram for $D_5$.  The word $v_K$ is $s_1s_3s_4s_5s_2$. 

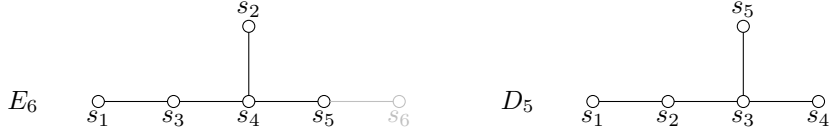
\begin{center}
\begin{tikzpicture}
\draw[] (0,0) circle [radius=0.08];
\node [below] at (0,0)  {$s_1$};
\draw[] (1,0) circle [radius=0.08];
\node [below] at (1,0)  {$s_3$};
\draw (0.08,0) -- (.92,0);
\draw (1.08,0) -- (1.92,0);
\draw[] (2,0) circle [radius=0.08];
\node [below] at (2,0)  {$s_4$};
\draw[] (3,0) circle [radius=0.08];
\node [below] at (3,0)  {$s_{5}$};
\draw [] (2.08,0) -- (2.92,0);
\draw[lightgray] (4,0) circle [radius=0.08];
\node [below] at (4,0)  {\color{lightgray}$s_{6}$};
\draw [lightgray] (3.08,0) -- (3.92,0);
\draw[] (2,1) circle [radius=0.08];
\node [above] at (2,1)  {$s_{2}$};
\draw [] (2,.08) -- (2,.92);
\node at (-1,0){$E_6$};
\end{tikzpicture}
\hspace{7mm}
\begin{tikzpicture}
\draw[] (0,0) circle [radius=0.08];
\node [below] at (0,0)  {$s_1$};
\draw[] (1,0) circle [radius=0.08];
\node [below] at (1,0)  {$s_2$};
\draw (0.08,0) -- (.92,0);
\draw (1.08,0) -- (1.92,0);
\draw[] (2,0) circle [radius=0.08];
\node [below] at (2,0)  {$s_3$};
\draw[] (3,0) circle [radius=0.08];
\node [below] at (3,0)  {$s_{4}$};
\draw [] (2.08,0) -- (2.92,0);
\draw[] (2,1) circle [radius=0.08];
\node [above] at (2,1)  {$s_{5}$};
\draw [] (2,.08) -- (2,.92);
\node at (-1,0){$D_5$};
\end{tikzpicture}
\end{center}
\end{Example}

%\begin{Definition}[Harada-Tymoczko~\cite{Monk}]
%\label{def:v sub k}
% For any subset of simple roots $K$ where $K=\lbrace \alpha_{a_1}, \alpha_{a_2}, \ldots , \alpha_{a_{m} } \rbrace$ and $a_1<a_2<\cdots <a_{m}$ define the word $v_K \in W$ by
%$$v_K=s_{a_1}s_{a_2} \cdots s_{a_{m}}.$$
%\end{Definition}
\noindent
Note that $v_K$ is a  Coxeter element of $W_K$.  Because of the labeling  imposed on the simple roots in Figure \ref{fig:root systems}, each subset $K$ of simple roots corresponds to exactly one word $v_K$.
\begin{Theorem}[Basis of Peterson Schubert classes]
\label{thm:basis}
The Peterson Schubert classes $\lbrace p_{v_K} : K\subseteq \Delta \rbrace$ are a basis of $H_{S^1}^*(Pet)$ as a module over the $S^1$-equivariant cohomology of a point, $H_{S^1}^*(pt)\cong \mathbb{C}[t]$.
\end{Theorem}
\noindent
This is a version of Harada-Tymoczko's Theorem 5.9 \cite{Poset Pinball}.  With Precup's work and Lemmas \ref{J subset K} and \ref{non-zero} we extend the proof to all Peterson varieties.
\begin{Lemma} \label{J subset K}
For any set of simple roots $\Delta$ and any subsets $J, K \subseteq \Delta$ the polynomial $p_{v_J}(w_K)$ is zero unless $J\subseteq K$.
\end{Lemma}
\begin{proof}
Suppose $J \not\subseteq K$ and that $\alpha_j \in J \setminus K$.  Then $s_j \leq v_J$ and $s_j\not\leq w_K$ in the Bruhat order.  For $\sigma_{v_J}(w_K)$ to be non-zero there must be some subword of $w_K$ that is equal to $v_J$ and therefore $v_J \leq w_K$.  But $s_j \leq v_J$ implies that $s_j \leq w_K$ which is a contradiction.  Thus $\sigma_{v_J}(w_K)=0$ by Property \ref{billey2} of Proposition \ref{prop:billey} and by construction $p_{v_J}(w_K)=0$.
  \end{proof}

\begin{Lemma} \label{non-zero}
For any set of simple roots $\Delta$ and any subset $K\subseteq \Delta$
$$p_{v_K}(w_K) \neq 0.$$
\end{Lemma}
\begin{proof}
Since $v_K \in W_K$ we must have $v_K \leq w_K$.  By Property \ref{billey3} of Proposition \ref{prop:billey} the polynomial $\sigma_{v_K}(w_K) \in \mathbb{C}[\alpha_i:\alpha_i\in \Delta]$ is not equal to zero.  We have defined $p_{v_K}(w_K)$ to be $\pi_1(\sigma_{v_K}(w_K))$.  Since $\sigma_{v_K}(w_K)$ has positive integer coefficients over the simple roots by Property\ref{billey4} of the same proposition, its image in $\mathbb{C}[t]$ must also have positive integer coefficients.
  \end{proof}

\begin{proof}[Basis of Peterson Schubert classes]
 Impose a partial order on the sets $\lbrace K\subseteq \Delta \rbrace$ by inclusion and extend that to a total order on the sets $K$ of simple roots.  Use that total order to order the classes $\lbrace p_{v_K}\rbrace$  and the $S^1$-fixed points $w_K \in Pet$.  Lemma ~\ref{J subset K} implies that the collection $\lbrace p_{v_K}\rbrace$  is lower-triangular and  Lemma ~\ref{non-zero} implies that the collection has full rank.  Thus $\lbrace p_{v_K}\rbrace$ is a linearly independent set. By the Property \ref{billey1} of Billey's formula, the polynomial degree of $p_{v_K}$ is $|K|$ and its cohomology degree is $2|K|$.  As there are ${n}\choose{|K|}$ subsets of $\Delta$ with size $|K|$, there are exactly ${n}\choose{|K|}$ Peterson Schubert varieties with cohomology degree $2|K|$.\\
\\
A paving by affines computes the Betti numbers of $Pet$ \cite[19.1.11]{Fulton Intersection Theory}. Precup's paving by affines reveals that the dimensions of the corresponding pavings are also $n\choose |K|$ \cite[Corollary 4.13]{Precup}.  As a linearly independent set with the right number of elements of each degree, the set $\lbrace p_{v_K}\rbrace$ is a module basis of $H_{S^1}^*(Pet)$ \cite[Proposition A.1]{Monk}.
%By Theorem~\ref{thm:spanning} it suffices to show that the classes $\lbrace p_{v_K} : K\subseteq \Delta \rbrace$ are linearly independent over $\mathbb{C}[t]$. 
  \end{proof}

\begin{Example}
Below is the Peterson Schubert class basis of the $S^1$-equivariant cohomology of $Pet$  in Lie type $C_3$. The classes and fixed points are indexed by the subsets $K\subseteq \Delta$.\\
\scalebox{.8}{
\begin{tabular}{ccccccccc}
$K$ & $p_{v_\emptyset}$ & $p_{v_{\lbrace \alpha_1 \rbrace}}$ & $p_{v_{\lbrace \alpha_2 \rbrace}}$ & $p_{v_{\lbrace \alpha_1 \rbrace}}$ & $p_{v_{\lbrace \alpha_1, \alpha_2 \rbrace}}$ & $p_{v_{\lbrace \alpha_1, \alpha_3 \rbrace}}$ & $p_{v_{\lbrace \alpha_2, \alpha_3 \rbrace}}$ & $p_{v_{\lbrace \alpha_1, \alpha_2, \alpha_3 \rbrace}}$ \\
$\begin{array}{c}
\emptyset \\
\lbrace \alpha_1 \rbrace \\
\lbrace \alpha_2 \rbrace \\
\lbrace \alpha_3 \rbrace \\
{\lbrace \alpha_1, \alpha_2 \rbrace} \\
{\lbrace \alpha_1, \alpha_3 \rbrace}\\
{\lbrace \alpha_2, \alpha_3 \rbrace}\\
\lbrace \alpha_1, \alpha_2, \alpha_3 \rbrace\\
\end{array}$
&
$\begin{pmatrix}
1\\
1\\
1\\
1\\
1\\
1\\
1\\
1\\
\end{pmatrix}$
&
$\begin{pmatrix}
0\\
t\\
0\\
0\\
2t\\
t\\
0\\
5t\\
\end{pmatrix}$
&
$\begin{pmatrix}
0\\
0\\
t\\
0\\
2t\\
0\\
3t\\
8t\\
\end{pmatrix}$&
$\begin{pmatrix}
0\\
0\\
0\\
t\\
0\\
t\\
4t\\
9t\\
\end{pmatrix}$&
$\begin{pmatrix}
0\\
0\\
0\\
0\\
2t^2\\
0\\
0\\
20t^2\\
\end{pmatrix}$&
$\begin{pmatrix}
0\\
0\\
0\\
0\\
0\\
t^2\\
0\\
45t^2\\
\end{pmatrix}$&
$\begin{pmatrix}
0\\
0\\
0\\
0\\
0\\
0\\
6t^2\\
36t^2\\
\end{pmatrix}$&
$\begin{pmatrix}
0\\
0\\
0\\
0\\
0\\
0\\
0\\
60t^3\\
\end{pmatrix}$

\end{tabular}}
\end{Example}
\noindent

\section{Monk's Formula}
With the Peterson Schubert class basis for $H_{S^1}^*(Pet)$ defined in Theorem \ref{thm:basis}, we can examine the structure of $H_{S^1}^*(Pet)$ through its multiplication rules.  First we determine a minimal set of Peterson Schubert classes that generate the ring $H_{S^1}^*(Pet)$.

\begin{Lemma}
The Peterson Schubert classes $p_{s_i}$  are a minimal generating set for the ring $H_{S^1}^*(Pet)$ over $H_{S^1}^*(pt)$. 
\end{Lemma}
\begin{proof}
That the classes $p_{s_i}$ generate $H_{S^1}^*(Pet)$ is a consequence of Theorem \ref{thm:by type}. Each class $p_{s_i}$ has polynomial degree one, so if $p_{s_j}$ can be expressed in terms of the other degree-one Peterson Schubert classes, it is a sum $$p_{s_j}=\sum \limits_{i\neq j} a_i \cdot p_{s_i}$$ for some coefficients $a_i\in \mathbb{C}$.  But Theorem \ref{thm:basis} shows that the Peterson Schubert classes are linearly independent, so the generating set $\{p_{s_i}: \alpha_i\in \Delta \}$ is minimal.
%It is well known that the Schubert classes $\sigma_{s_i}$ generate $H_T^*(G/B)$ \cite[Theorem 1]{On Equivariant Cohomology}  By Theorem ~\ref{thm:spanning} the  map $$\phi: H_T^*(G/B) \to H_{S^1}^*(Pet)$$ is a surjective ring homomorphism.  Fix $\beta\in  H_{S^1}^*(Pet)$ and let $\gamma \in  H_{T}^*(G/B)$ be in the preimage of $\beta$.  Then for some set of constants $c_I\in \mathbb{C}$ the element $\gamma=\sum \limits_{I\subseteq \Delta} c_I \cdot \prod \limits_{i \in I} \sigma_{s_i}$ where $I$ can be a multiset.  Thus
%$$\beta=\phi( \gamma )=\phi \left( \sum \limits_{I\subseteq \Delta} c_I \cdot \prod \limits_{i \in I} \sigma_{s_i} \right)=\sum \limits_{I\subseteq \Delta} c_I \cdot \prod \limits_{i \in I} \phi(\sigma_{s_i})=\sum \limits_{I\subseteq \Delta} c_I \cdot \prod \limits_{i \in I_\gamma} p_{s_i}.$$
   \end{proof}

\noindent
Monk's rule is an explicit formula for multiplying an arbitrary module generator class $p_{v_K}$  by a ring generator class $p_{s_i}$.  For the Peterson variety, Monk's formula gives a set of constants $c_{i,K}^J \in H_{S^1}^*(\text{pt})$ such that 
\begin{equation}
\label{eq:general monk}
p_{s_i}p_{v_K}=\sum \limits_{J\subseteq \Delta} c_{i,K}^J \cdot p_{v_J}.
\end{equation}
\noindent
The Peterson Schubert classes $\lbrace p_{v_K}: K\subseteq \Delta \rbrace$ are a module basis for $H_{S^1}^*(Pet)$ and the product of $p_{s_i}$ and $p_{v_K}$ is also in that module. Thus a unique set of constants $\lbrace c_{i,K}^J \rbrace$ solve this equation.  Because $H_{S^1}^*(pt)=\mathbb{C}[t]$ these structure constants are complex polynomials in $t$.

\begin{Theorem}[Monk's formula for Peterson varieties]
\label{thm:Monk formula}
The Peterson Schubert classes satisfy 
$$p_{s_i} \cdot p_{v_K} = p_{s_i}(w_K) \cdot p_{v_K} + \sum \limits_{\substack{J {\text{ such that}}\\
K\subseteq J \subseteq \Delta \\ |J|=|K|+1}} c_{i,K}^J \cdot p_{v_J}$$
where the coefficients $c_{i,K}^J$ are non-negative rational numbers and  $$c_{i,K}^J=(p_{s_i}(w_J)-p_{s_i}(w_K)) \cdot \frac{p_{v_K}(w_J)}{p_{v_J}(w_J)}.$$
\end{Theorem}  
\noindent
We use two lemmas to eliminate many subsets $J\subseteq \Delta$ by showing that $c_{i,K}^J=0$. 

\begin{Lemma} \label{lemma J>K+1}
If $|J|>|K|+1$ then $c_{i,K}^J=0$.
\end{Lemma}
\begin{proof}
The polynomial degree of $p_{v} \in \mathbb{C}[t]$ is the length of a reduced word for $v$. Therefore the Peterson Schubert class $p_{v_K}$ has degree $|K|$ and the polynomial degree of $p_{s_i}p_{v_K}$ is $|K|+1$.  The polynomial degrees on the right- and left-hand sides of Equation~\eqref{eq:general monk} must be equal.  Take only the parts of each side of Equation~\eqref{eq:general monk} that have degree higher than $|K|+1$. Hence it follows that
$$
0=\sum \limits_{\substack{J\subseteq \Delta \\|J|>|K|+1}} c_{i,K}^J\cdot p_{v_J}.$$
The Peterson Schubert classes $p_{v_J}$  are linearly independent by Theorem~\ref{thm:basis}.  Therefore whenever $|J|>|K|+1$ the coefficient $c_{i,K}^J=0$.
    \end{proof}
%\noindent
%We can further refine Equation~\eqref{eq:general monk} by removing another set of  subsets $J\subseteq \Delta$  from the equation.

\begin{Lemma} \label{lemma K subset J}
The constant $c_{i,K}^J=0$ unless $K\subseteq J$.
\end{Lemma}
\begin{proof}
Suppose that $L$ is the smallest counter example, i.e., $L\subseteq \Delta$ does not contain $K$ and for all $H\subsetneq L$ the coefficient $c_{i,K}^H=0$.  Evaluate Monk's formula at the $S^1$-fixed point $w_L$ to get
$$p_{s_i}(w_L) \cdot p_{v_K}(w_L)=\sum \limits_{\substack{J\subseteq \Delta\\ |J| \leq |K|+1}} c_{i,K}^J \cdot p_{v_J}(w_L).$$
\noindent
The word $v_K \not\leq w_L$ by hypothesis so the left-hand side is $0$.   If $J \not\subseteq L$ then $p_{v_J}(w_L)=0$ and thus
$$0=\sum \limits_{\substack{J\subseteq L \subseteq \Delta\\ |J| \leq |K|+1}} c_{i,K}^J \cdot p_{v_J}(w_L).$$
\noindent
By construction if $J\subsetneq L$ then $c_{i,K}^J=0$ so we are left with $$0= c_{i,K}^L \cdot p_{v_L}(w_L).$$
\noindent
By Lemma~\ref{non-zero} the evaluation $p_{v_L}(w_L) \neq 0$.  Since $H_{S^1}^*(pt)=\mathbb{C}[t]$ is an integral domain we conclude that $c_{i,K}^L =0$.
   \end{proof}

\noindent
Having determined which coefficients are zero, we give a third lemma addressing the non-zero coefficients.
\begin{Lemma}
\label{lemma ct}
Consider the map $\pi_1:H_T^*(G/B) \to H_{S^1}^*(G/B)$ from Equation ~\eqref{eq:map diagram}. Let $v,w$ be elements of the Weyl group.
The image under $\pi_1$ of the evaluation $\sigma_v(w)$ of a Schubert class $\sigma_v$ at the fixed point $w$ is a monomial $c\cdot t^m$ where $c$ is a non-negative integer and $m$ is the length of $v$.
\end{Lemma}

\begin{proof}
By the Properties \ref{billey1} and \ref{billey4} of Billey's formula given in Proposition~\ref{prop:billey}, the polynomial $\sigma_v(w)$ is homogeneous of degree $\ell(v)$ with non-negative integer coefficients.  Its image $\pi_1(\sigma_v(w))$ is $ct^{\ell(v)}$ where $c$ is the sum of the integer coefficients of $\sigma_v(w)$.  
   \end{proof}
\noindent
We now prove Theorem \ref{thm:Monk formula}.
\begin{proof}[Monk's formula for Peterson varieties]
By Lemma \ref{lemma J>K+1}  the general Monk's formula in Equation~\eqref{eq:general monk} simplifies to 
$$
p_{s_i} \cdot p_{v_K} = \sum \limits_{\substack{|J|\leq|K|+1}} c_{i,K}^J \cdot p_{v_J}
$$
and Lemma \ref{lemma K subset J} further refines the equation to
\begin{equation}
\label{eq:pre-monk}
p_{s_i} \cdot p_{v_K} = c_{i,K}^K \cdot p_{v_K} + \sum \limits_{\substack{K\subsetneq J \subseteq \Delta \\ |J|=|K|+1}} c_{i,K}^J \cdot p_{v_J}.
\end{equation}
\noindent
We evaluate both sides of Equation~\eqref{eq:pre-monk} at the $S^1$-fixed point $w_K$ and use the fact that $p_{v_J}(w_K)=0$ whenever $J$ is not a subset of $K$ to obtain 
$$p_{s_i}(w_K) \cdot p_{v_K}(w_K) = c_{i,K}^K \cdot p_{v_K}(w_K).$$
\noindent
The polynomial $ p_{v_K}(w_K)$ is non-zero by Lemma \ref{non-zero}. Since $\mathbb{C}[t]$ is an integral domain we may divide both sides by $ p_{v_K}(w_K)$. This leaves $ c_{i,K}^K  = p_{s_i}(w_K).$  By Lemma~\ref{lemma ct} the polynomial $p_{s_i}(w_K)$ is a degree-one monomial with an integer coefficient. \\
\\
Next fix a subset $L\subseteq \Delta$ such that $K\subsetneq L $ and let $|L|=|K|+1$.  Evaluating at the $S^1$-fixed point $w_L$ gives
$$p_{s_i}(w_L) \cdot p_{v_K}(w_L) = p_{s_i}(w_K) \cdot p_{v_K}(w_L) + \sum \limits_{\substack{J \text{ such that}\\ K\subsetneq J \subseteq \Delta \\ |J|=|K|+1}} c_{i,K}^J \cdot p_{v_J}(w_L).$$
\noindent
But $p_{v_J}(w_L)=0$ unless $J\subseteq L$ by the Properties \ref{billey2} and \ref{billey3} of Billey's formula so in fact 
$$p_{s_i}(w_L) \cdot p_{v_K}(w_L) = p_{s_i}(w_K) \cdot p_{v_K}(w_L) + c_{i,K}^L \cdot p_{v_L}(w_L).$$
\noindent Solving for $c_{i,K}^L$ gives
\begin{equation}
\label{eq:structure constants}
c_{i,K}^L=(p_{s_i}(w_L)-p_{s_i}(w_K))\cdot \frac{p_{v_K}(w_L)}{p_{v_L}(w_L)}.
\end{equation}
\noindent
If the term $(p_{s_i}(w_L)-p_{s_i}(w_K))=0$ then the constant $c_{i,K}^L$ is non-negative and rational.  Suppose that $(p_{s_i}(w_L)-p_{s_i}(w_K))\neq 0$.  By Lemma~\ref{lemma ct} $(p_{s_i}(w_L)-p_{s_i}(w_K))$ has degree one.  By the same lemma  $\frac{p_{v_K}(w_L)}{p_{v_L}(w_L)}$ has degree $|K|-|L|=|K|-(|K|+1)=-1$.  Thus $c_{i,K}^L$ is a priori a rational number. It remains to show that $c_{i,K}^L$ is non-negative.\\
\\
It suffices to show that $(p_{s_i}(w_L)-p_{s_i}(w_K))$ is non-negative because $\frac{p_{v_K}(w_L)}{p_{v_L}(w_L)}$ will always be non-negative.  The word $w_L$ can be written as $w_K\cdot \tilde{w}$ for some reduced word $\tilde{w} \in W_L$ \cite{Bjorner and Brenti}.   Let $s_{b_1}s_{b_2}\cdots s_{b_{m}}$ be a reduced word for $w_K$ and $s_{b_{m+1}} s_{b_{m+2}} \cdots s_{b_n}$ be a reduced word for $\tilde{w}$.  The length $\ell(s_i)=1$ for each $i$ so Billey's formula says 
$$\begin{array}{rl}
\sigma_{s_i}(w_K \cdot \tilde{w})&=
 \sum \limits_{{s_{b_j}=s_i} }   \mathbf{r}(\mathbf{j},w_L) \\
\\
&=
 \sum \limits_{\substack{{s_{b_j}=s_i} \\ j\leq m}}   \mathbf{r}(\mathbf{j},w_L) +  \sum \limits_{\substack{{s_{b_j}=s_i} \\ j> m}}   \mathbf{r}(\mathbf{j},w_L) \\
\\
&=
\sigma_{s_i}(w_K) +  \sum \limits_{\substack{{s_{b_j}=s_i} \\ j> m}}   \mathbf{r}(\mathbf{j},w_L) .\end{array}$$
\noindent
Since $\pi_1$ is a ring homomorphism from $\mathbb{C}[\alpha_i:\alpha_i\in \Delta]$ to $\mathbb{C}[t]$, we obtain
$$p_{s_i}(w_J)-p_{s_i}(w_K)=\pi_1\left(\sigma_{s_i}(w_J)-\sigma_{s_i}(w_K)\right)=\pi_1\left(\sum \limits_{\substack{{s_{b_j}=s_i} \\ j> m}}   \mathbf{r}(\mathbf{j},w_J)\right).$$
\noindent
By the definition of Billey's formula each term $ \mathbf{r}(\mathbf{j},w_L) $ is a positive root in $\Phi$.  Therefore its image $\pi_1(\mathbf{r}(\mathbf{j},w_J))$ is $c t$ for some positive integer $c$.  The $t$ is canceled by $\frac{p_{v_K}(w_L)}{p_{v_L}(w_L)}$ which has degree $-1$.  Thus $(p_{s_i}(w_L)-p_{s_i}(w_K))$ is non-negative and so is the coefficient $c_{i,K}^J$.
   \end{proof}
\noindent
 In classical Schubert calculus the structure constants are generally non-negative integers. Frequently they are in bijection with dimensions of irreducible representations.  However, structure constants for the Peterson variety are {\em not} necessarily integers.  For example in type $D_5$ let $K=\lbrace \alpha_1, \alpha_2,\alpha_3, \alpha_4 \rbrace$ and $J=\Delta$. Then
$$c_{5,K}^J=\frac{5}{2}.$$
\noindent
\begin{Conjecture}
We conjecture that in this basis, non-integral structure constants only occur in Lie types $D$ and $E$.  
\end{Conjecture}

\section{Giambelli's Formula}
Giambelli's formula expresses an arbitrary module-basis element in terms of the ring generators.  In both ordinary and equivariant cohomology many spaces have determinental Giambelli's formulae. % For the basis of $H^*(Flags)$ consisting of Schubert classes Giabmell's formula is
%$$ \sigma_\lambda = \det( \sigma_{\lambda_i+j-i})_{1\leq i,j \leq r}$$
%\noindent
%where $\sigma_\lambda$ is the Schubert class corresponding to the partition $\lambda=(\lambda_1,\ldots,\lambda_r)$ \cite{Fulton Young Tableaux}.  While we might expect and equivariant formula to have a different form, the equivariant Giambelli's formulas for Grassmannians and partial flag varieties are are also determinental formulas.  
Giambelli's formula for Peterson varieties, however, simplifies to a single product. 
\noindent
\begin{Lemma} 
\label{lemma one product}
For any Peterson Schubert class $p_{v_K}$ there exists a constant $C$ satisfying
\begin{equation}
\label{eq:Giambelli}
C \cdot p_{v_K}=\prod \limits_{\alpha_i \in K} p_{s_i}.
\end{equation}
\end{Lemma}

\begin{proof}
If $|K|=m$ let $K=\lbrace \alpha_{a_1}, \alpha_{a_2}, \cdots \alpha_{a_{m}} \rbrace$. Define a sequence of nested subsets $\emptyset = K_0 \subsetneq K_1 \subsetneq K_2 \subsetneq \cdots \subsetneq K_{m}=K$  by $$K_i=\lbrace \alpha_{a_1}, \alpha_{a_2}, \cdots \alpha_{a_{i}} \rbrace.$$
\noindent
From Equation~\eqref{eq:pre-monk} in the proof of Monk's formula for Peterson varieties
$$p_{s_{a_{i+1}}}\cdot p_{v_{K_i}}=  c_{a_{i+1},K_i}^{K_{i+1}}\cdot p_{v_{K_i}} + \sum \limits_{\substack{K_i\subsetneq J \subseteq \Delta \\ |J|=|K_i|+1}} c_{a_{i+1},K_i}^J \cdot p_{v_J}.$$
\noindent
Theorem~\ref{thm:Monk formula} says $c_{a_{i+1},K_i}^{K_{i+1}}= p_{s_{a_{i+1}}}(w_{K_i})$. Because $\alpha_{a_{i+1}} \not \in K_i$ the coefficient $p_{s_{a_{i+1}}}(w_{K_i})=0$ . If $\alpha_{a_{i+1}} \not\in J$ the term $ p_{s_{a_{i+1}}}(w_{J})=0$. Thus if  $J\neq K_{i+1}$ the coefficient $c_{a_{i+1},K_i}^J =0$. Now Equation~\eqref{eq:pre-monk} reduces to
$$p_{s_{a_{i+1}}}\cdot p_{v_{K_i}}=  c_{a_{i+1},K_i}^{K_{i+1}} \cdot p_{v_{K+1}}.$$
\noindent
Solving for $p_{v_{K+1}}$ gives 
$$\frac{p_{s_{a_{i+1}}}\cdot p_{v_{K_i}} }{ c_{a_{i+1},K_i}^{K_{i+1}}} = p_{v_{K{i+1}}}.$$
By induction on $i$ we see 
$$p_{v_K}=\frac{\prod \limits_{i=1}^{|K|} p_{s_{a_{i}}}} {\prod \limits_{i=1}^{|K|} c_{a_{i},K_{i-1}}^{K_{i}}}.$$
This gives that $$C=\prod \limits_{i=1}^{|K|} c_{a_{i+1},K_i}^{K_{i+1}}.$$
   \end{proof}
\noindent
To find this constant $C$ explicitly we consider the simplest non-trivial Peterson Schubert classes, those that are connected.

\begin{Definition}
From Definition \ref{def: connected part 1}, a subset of simple roots $K\subseteq \Delta$ is called { connected} if the induced Dynkin diagram of $K$ is a connected subgraph of the Dynkin diagram of $\Delta$.  The class $p_{v_K}$ is called {connected} whenever $K$ is connected.
\end{Definition}

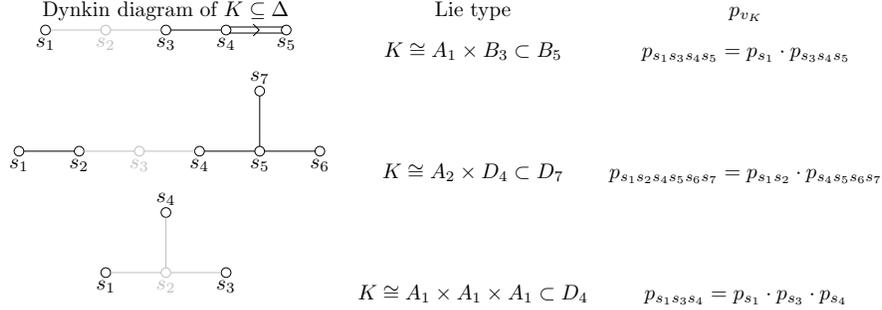
\begin{figure}

\begin{center}
\scalebox{.8}{\begin{tabular}{ccc}
  Dynkin diagram of $K \subseteq \Delta$ & Lie type  & $p_{v_K}$\\ 
\begin{tikzpicture}[scale=1.0]
\draw[] (1,0) circle [radius=0.08];
\node [below] at (1,0)  {$s_1$};
\draw[lightgray] (2,0) circle [radius=0.08];
\node [below]  at (2,0)  {${\color{lightgray}s_2}$};
\draw [lightgray] (1.08,0) -- (1.92,0);
\draw[] (3,0) circle [radius=0.08];
\node [below] at (3,0)  {$s_{3}$};
\draw [lightgray] (2.08,0) -- (2.92,0);
\draw[] (4,0) circle [radius=0.08];
\node [below] at (4,0)  {$s_4$};
\draw (3.08,0) -- (3.92,0);
\draw[] (5,0) circle [radius=0.08];
\node [below] at (5,0)  {$s_{5}$};
\draw (4.05,.05) -- (4.95,.05);
\draw (4.05,-.05) -- (4.95,-.05);
\draw (4.45,.1) -- (4.55,0) -- (4.45,-.1);
\end{tikzpicture}
& $K\cong A_1 \times B_3 \subset B_5$& $p_{s_1s_3s_4s_5}=p_{s_1} \cdot p_{s_3s_4s_5}$
\\\
 
\begin{tikzpicture}[scale=1.0]
\draw[] (0,0) circle [radius=0.08];
\node [below] at (0,0)  {$s_1$};
\draw[] (1,0) circle [radius=0.08];
\node [below] at (1,0)  {$s_2$};
\draw (0.08,0) -- (.92,0);
\draw[lightgray] (2,0) circle [radius=0.08];
\node [below] at (2,0)  {${\color{lightgray}s_3}$};
\draw [lightgray] (1.08,0) -- (1.92,0);
\draw[] (3,0) circle [radius=0.08];
\node [below] at (3,0)  {$s_{4}$};
\draw [lightgray] (2.08,0) -- (2.92,0);
\draw[] (4,0) circle [radius=0.08];
\node [below] at (4,0)  {$s_{5}$};
\draw (3.08,0) -- (3.92,0);
\draw[] (5,0) circle [radius=0.08];
\node [below] at (5,0)  {$s_{6}$};
\draw (4.08,0) -- (4.92,0);
\draw[] (4,1) circle [radius=0.08];
\node [above] at (4,1)  {$s_7$};
\draw (4,.08) -- (4,.92);
\end{tikzpicture}
& $K\cong A_2 \times D_4 \subset D_7$& $p_{s_1s_2s_4s_5s_6s_7}=p_{s_1s_2} \cdot p_{s_4s_5s_6s_7}$
\\

\begin{tikzpicture}[scale=1.0]
\draw[] (3,0) circle [radius=0.08];
\node [below] at (3,0)  {$s_{1}$};
\draw[lightgray] (4,0) circle [radius=0.08];
\node [below] at (4,0)  {${\color{lightgray}s_{2}}$};
\draw [lightgray](3.08,0) -- (3.92,0);
\draw[] (5,0) circle [radius=0.08];
\node [below] at (5,0)  {$s_{3}$};
\draw [lightgray](4.08,0) -- (4.92,0);
\draw[] (4,1) circle [radius=0.08];
\node [above] at (4,1)  {$s_4$};
\draw [lightgray] (4,.08) -- (4,.92);
\end{tikzpicture} 
& $K\cong A_1 \times A_1 \times A_1 \subset D_4$
& $p_{s_1s_3s_4}=p_{s_1}\cdot p_{s_3} \cdot p_{s_4}$
\\
\end{tabular}}\\
\caption{The subsystem $K\subseteq \Delta$ is drawn as a marked set of vertices in the Dynkin diagram. The associated Peterson Schubert class is given at right.} \label{fig:root subsystems}
\end{center}
\end{figure}

\noindent
Figure \ref{fig:root subsystems} gives examples of sets $K\subseteq \Delta$ which are not connected.  The induced Dynkin diagrams also give the Lie type of the subsystem $K$.  Every Peterson Schubert class can be expressed in terms of connected classes.

\begin{Theorem}\label{thm:disconnected}
If $J,K\subset \Delta$  are each connected subsets such that $J\cup K$ is disconnected then \begin{equation} \label{eq: disconnected classes} p_{v_{J \cup K}}=p_{v_J}\cdot p_{v_K}. \end{equation}
\end{Theorem}
\begin{proof}
We show that equality holds when Equation \eqref{eq: disconnected classes} is evaluated at any $S^1$-fixed point $w_L$.  If $L$ does not contain $J\cup K$ we can suppose without loss of generality that $J\not\subseteq L$.  Then both $p_{v_{J \cup K}}(w_L)$ and $p_{v_J}(w_L)$ are zero.\\
\\
Now suppose $J \cup K \subseteq L$. Even though $J\cup K$ is disconnected, $L$ may be connected or disconnected.  Fix a reduced word for $w_L$
$$\tilde{w}_L=s_{a_1}s_{a_2}\cdots s_{a_{\ell(w_L)}}$$
and let $b \prec \tilde{w}_L$ mean that $b$ is a subword of $\tilde{w}_L$.  The indexing set of the subword $b$ is the set $I(b) \subseteq \lbrace 1,2, \ldots \ell(w_L) \rbrace$ such that $$b=s_{a_{j_1}}s_{a_{j_2}} \cdots s_{a{j_{|I(b)|}}}\text{ for } j_1<j_2<\cdots <j_{|I(b)|} \text{ with each }j_k \in I(b).$$
The subwords of $\tilde{w}_L$ are in bijection with the subsets of $\lbrace 1,2, \ldots \ell(w_L) \rbrace$.  Given two subwords $b_1,b_2 \prec \tilde{w}_L$ we define their union $b_1 \cup b_2$ to be the subword
$$b_1 \cup b_2 =s_{a_{j_1}}s_{a_{j_2}} \cdots s_{a{j_{|I(b_1)\cup I(b_2)|}}}$$
 for  $j_1<j_2<\cdots <j_{|I(b_1)\cup I(b_2)|}$ with each $j_k \in I(b_1)\cup I(b_2).$  Let $b_J, b_K \prec \tilde{w}_L$ be reduced words for $v_J$ and $v_K$ respectively.  Since $J$ and $K$ are disconnected $I(b_J) \cap I(b_K) = \emptyset$ and $v_J$ commutes entirely with $v_K$ \cite{Bjorner and Brenti}. Thus $b_J \cup b_K$ is a reduced word for $v_J\cdot v_K=v_{J\cup K}$.\\
\\
Conversely let $b\prec \tilde{w}_L$ be a reduced word for $v_{J\cup K}$.  We can partition $I(b)$ into $$I(b)_J=\lbrace j_k\in I(b) : \alpha_{a_{j_k}}\in J \rbrace \text{ and }I(b)_K=\lbrace j_k\in I(b) : \alpha_{a_{j_k}}\in K \rbrace.$$
\noindent
Since $v_J \leq v_{J\cup K}$ and $b$ is a reduced word for $v_{J\cup K}$, some reduced word for $v_J$ must be a subword of $b$.  Let $b_J \prec b$ be that subword.  Since no reflections associated with $K$ are in $b_J$, $I(b_J)\subseteq I(b)_J$.  A parallel argument shows that there is some subword $b_K\prec b$ equal to $v_K$ and that $I(b_K)\subseteq I(b)_K$.\\
\\
By our previous argument $b_J\cup b_K$ is a reduced word for $v_J\cdot v_K =v_{J\cup K}$.  So $\ell(b_J\cup b_K)=\ell(v_{J\cup K})$ which equals $\ell(b)$. Thus $I(b_J)= I(b)_J$ and $I(b_K)= I(b)_K$ and $b=b_J \cup b_K$.\\
\\
A subword $b\prec \tilde{w}_L$ is a reduced word for $v_{J\cup K}$ if and only if $b=b_J \cup b_K$ for subwords $b_J, b_K \prec \tilde{w}_L$ reduced words for $v_J$ and $v_K$.  Billey's formula in Equation~\eqref{eq:Billeys} is a sum over such subwords.  We use it to rewrite the left- and right-hand sides of Equation~\eqref{eq: disconnected classes}. The left-hand side becomes:
\begin{equation}\label{eq:lhs of disconnect}
p_{v_{J\cup K}}(w_L)= \sum \limits_{\substack{b \prec \tilde{w}_L \\ b= v_{J\cup K} \\ |I(b)|=|J \cup K|}} \left( \prod \limits_{j\in I(b)}  \mathbf{r}(\mathbf{j},\tilde{w}_L) \right).
\end{equation}
\noindent
Similarly the right-hand side becomes $p_{v_{J}}(w_L)\cdot p_{v_{K}}(w_L)=$
\begin{equation}\label{eq:rhs of disconnect}
\left[ \sum \limits_{\substack{b \prec \tilde{w}_L \\ b= v_{J} \\ |I(b)|=|J |}} \left( \prod \limits_{j\in I(b)}  \mathbf{r}(\mathbf{j},\tilde{w}_L) \right) \right] \cdot \left[\sum \limits_{\substack{b \prec \tilde{w}_L \\ b= v_{ K} \\ |I(b)|=|K|}} \left( \prod \limits_{j\in I(b)}  \mathbf{r}(\mathbf{j},\tilde{w}_L) \right)\right].
\end{equation}
\noindent
Both Equations \eqref{eq:lhs of disconnect} and \eqref{eq:rhs of disconnect} expand to the expression
$$
\sum \limits_{\substack{b_J \prec \tilde{w}_L \\ b_J= v_{J} \\ |I(b_J)|=|J |}} \sum \limits_{\substack{b_K \prec \tilde{w}_L \\ b_K= v_{K} \\ |I(b_K)|=|K|}} \left[  \left( \prod \limits_{j\in I(b_J)}  \mathbf{r}(\mathbf{j},\tilde{w}_L) \right) \cdot \left( \prod \limits_{j\in I(b_K)}  \mathbf{r}(\mathbf{j},\tilde{w}_L) \right)\right].
$$
   \end{proof}

\noindent
Any subset $K\subset \Delta$ gives rise to a Peterson Schubert class that is the product of connected Peterson Schubert classes. Understanding the connected Peterson Schubert classes thus gives full information on all Peterson Schubert classes.  The next theorem gives Giambelli's formula explicitly for connected Peterson Schubert classes. 

\begin{Theorem} \label{thm:by type}
If $K\subseteq \Delta$ is a connected root subsystem  of type $A_n,B
_n,C_n,F_4,$ or $G_2$ then $$|K|!\cdot p_{v_K} = \prod \limits_{\alpha_i \in K} p_{s_i}.$$
\noindent
If $K$ is a connected root subsystem of type $D_n$ then $$\frac{|K|!}{2}\cdot p_{v_K} = \prod \limits_{\alpha_i \in K} p_{s_i}.$$
If $K$ is a connected root subsystem of type $E_n$ then $$\frac{|K|!}{3}\cdot p_{v_K} = \prod \limits_{\alpha_i \in K} p_{s_i}.$$
\end{Theorem}
\noindent
Our proof of this theorem, given in Section 7, is combinatorial  and treats each Lie type as its own case.   Theorem \ref{thm:by type} can be restated uniformly.
\begin{Theorem}\label{thm:reduced words}
If $K\subseteq \Delta$ is a connected root subsystem of any Lie type and $|\mathcal{R}(v_K)|$ is the number of reduced words for $v_K$ then
$$\frac{|K|!}{|\mathcal{R}(v_K)|}\cdot p_{v_K} = \prod \limits_{\alpha_i \in K} p_{s_i}.$$
\end{Theorem}
\noindent
The uniformity across Lie types suggests that a uniform proof exists. Such a proof might shed light on the topology of these varieties. 
\begin{proof}
Given Theorem \ref{thm:by type} it is sufficient to show that $|\mathcal{R}(v_K)|=1$ if $K$ is type $A,B,C,F,$ or $G$, that $|\mathcal{R}(v_K)|=2$ for type $D$ and that $|\mathcal{R}(v_K)|=3$ for type $E$.  Given one reduced word any other reduced word can be obtained by a series of braid moves and commutations \cite{Bjorner and Brenti}.  If $K$ is type $A,B,C,F,$ or $G$ then $s_i$ and $s_{i+1}$ do not commute for any $i$.  Therefore $s_1s_2\cdots s_{n-1}s_n$ is the only reduced word for $v_K$.  \\
\\
If $K$ is of type $D$ then $s_i$ and $s_{i+1}$ commute if and only if $i=n-1$.  Also $s_{n-2}$ and $s_n$ do not commute.  The only reduced words for $v_K$ are $s_1s_2\cdots s_{n-2}s_{n-1}s_n$ and  $s_1s_2\cdots s_{n-2}s_{n}s_{n-1}$ so $|\mathcal{R}(v_K)|=2$. \\
\\
 If $K$ is type $E_n$ then we start with the word $v_K=s_1s_2s_3s_4\cdots s_n$ with the labels given as in Figure \ref{fig:root systems}.  The reflection $s_2$ commutes with $s_1$ and $s_3$ but not $s_4$. The reflection $s_3$ does not commute with $s_1$. When $i>2, s_i$ and $s_{i+1}$ do not commute.  Thus $v_K$ has exactly $3$ reduced words: $s_1s_2s_3s_4\cdots s_n$ and $s_1s_3s_2s_4\cdots s_n$ and $s_2s_1s_3s_4\cdots s_n$.
   \end{proof}
\noindent
We can now give Giambelli's formula explicitly for all Peterson Schubert classes.
\begin{Corollary}
If $K\subseteq \Delta$ and $K=K_1 \times K_2 \times \cdots K_j$ where each $K_\ell$ is maximally connected then 
\begin{equation}
C_K \cdot p_{v_K} = \prod \limits_{i\in K} p_{s_i}
\end{equation}
\noindent
where $C_K=\prod \limits_{\ell=1}^j \frac{|K_\ell|!}{|\mathcal{R}(v_{K_\ell})|}$.
\end{Corollary}
\begin{proof}
The corollary follows immediately from Theorems \ref{thm:disconnected} and \ref{thm:reduced words}.
   \end{proof}

\section{Modified Excited Young Diagrams}
To compute the constant term of Giambelli's formula we need to evaluate the Peterson Schubert class $p_{v}$ at $S^1$-fixed points of $Pet$.  Our main tool, modified excited Young diagrams, is related both to the work by Woo-Yong~\cite[Section 3]{Pipe Dreams}, and Ikeda-Naruse \cite{EYD}. \\
\\
We only need to evaluate at fixed points $w_K$ where $K\subseteq \Delta$ is a connected root system. For the remainder of this paper, $K$ will be a connected root system identified by Lie type. \\
\\
The first step is to write $w_K$ explicitly as a skew diagram $\lambda_{w_K}$.  The $i^{th}$ column from the left represents the simple reflection $s_i$.  Reading left-to-right and top-to-bottom gives a reduced word for $w_K$. Figure \ref{fig:the word w_K} gives several examples.\\
\\ 
Our goal is to compute $p_v(w_K)$. To start we use Equations \eqref{eq:Billeys} and \eqref{eq:map diagram} to rewrite it as
$$\begin{array}{rl}
p_v(w_K)=\pi_1 (\sigma_v(w_K))& =\pi_1 \left( \sum \limits_{\substack{\text{reduced words} \\ v=s_{b_{j_1}}s_{b_{j_2}} \cdots s_{b_{j_{\ell(v)}}} }} \left( \prod \limits_{i=1}^{\ell(v)}  \mathbf{r}(\mathbf{j_i},w_K) \right)\right)\\
\\
&=\sum \limits_{\substack{\text{reduced words} \\ v=s_{b_{j_1}}s_{b_{j_2}} \cdots s_{b_{j_{\ell(v)}}} }} \left( \prod \limits_{i=1}^{\ell(v)} \pi_1 ( \mathbf{r}(\mathbf{j_i},w_K)) \right).\end{array}$$
\noindent
To help with this computation we build two labeled diagrams. The first is called $\lambda_{w_K}$ and has boxes labeled by simple reflections. The second diagram is called $\lambda_{p(w_K)}$ has boxes labeled by integers. We label the $\mathbf{i}^{th}$ box of $\lambda_{w_K}$ with the value $\frac{1}{t}\cdot \pi_1(\mathbf{r}(\mathbf{i},w_K))$.  The term $\pi_1(\mathbf{r}(\mathbf{i},w_K))$ is a degree-one monomial in $\mathbb{C}[t]$ whose coefficient is the number of simple roots, counting multiple occurrences, that appear in the root $\mathbf{r}(\mathbf{i},w_K)$. This number is the height in the root poset of the root $\mathbf{r}(\mathbf{i},w_K)$. Thus the labels are positive integers. We give an example in type $B_3$:\\
\ytableausetup
{mathmode, boxsize=1em}
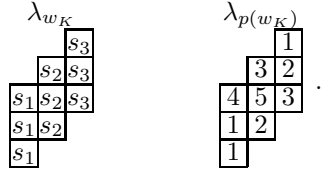
$$\begin{array}{ccc}
\lambda_{w_K} & &\lambda_{p(w_K)} \\
\begin{ytableau}
\none & \none & {s_3} \\
\none & s_2 & s_3\\
 {s_1} &  {s_2} &  {s_3} \\
 {s_1} & {s_2} \\
 {s_1} 
\end{ytableau}
&
\hspace{1cm}
&
\begin{ytableau}
\none & \none &  1\\
\none & 3 &2\\
4& 5 & 3 \\
 1&2 \\
 1
\end{ytableau}
\end{array}.$$

\noindent A $v$-excitation $\mu$ of $\lambda_{w_K}$ is any collection of $\ell(v)$ boxes in the labeled diagram that, when read left-to-right and top-to-bottom, give a reduced word for $v$.  For example if $K$ is type $B_3$ there are three $s_1s_2$-excitations of $\lambda_{w_K}$:
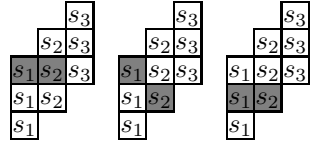
$$\begin{array}{ccc}
\begin{ytableau}
\none & \none & {s_3} \\
\none & s_2 & s_3\\
*(gray) {s_1} & *(gray) {s_2} &  {s_3} \\
 {s_1} & {s_2} \\
 {s_1} 
\end{ytableau}
&\begin{ytableau}
\none & \none & {s_3} \\
\none & s_2 & s_3\\
 *(gray){s_1} &  {s_2} &  {s_3} \\
 {s_1} & *(gray){s_2} \\
 {s_1} 
\end{ytableau}
&\begin{ytableau}
\none & \none & {s_3} \\
\none & s_2 & s_3\\
 {s_1} &  {s_2} &  {s_3} \\
 *(gray){s_1} & *(gray){s_2} \\
 {s_1} 
\end{ytableau}
\end{array}$$

\noindent
If $\mu$ is a $v$-excitation of $\lambda_{w_K}$ then $M_p(\mu)$ is the product of the entries in the boxes of $\lambda_{p(w_K)}$ filled by $\mu$.
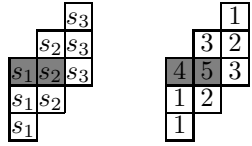
$$
\begin{ytableau}
\none & \none & {s_3} \\
\none & s_2 & s_3\\
*(gray) {s_1} & *(gray) {s_2} &  {s_3} \\
 {s_1} & {s_2} \\
 {s_1} 
\end{ytableau}
\hspace{1cm}
\begin{ytableau}
\none & \none & 1 \\
\none & 3 & 2\\
*(gray) 4 & *(gray) 5&  3 \\
1 & 2 \\
1 
\end{ytableau}
$$
\noindent For this $s_1s_2$-excitation $\mu$ of $w_{B_3}$ the coefficient is $M_p(\mu)=(4)(5)=20. $
Now $p_v(w_K)$ can be computed by this diagramatic construction: 
\begin{equation}
\label{eq:diagramatic billey}
 p_v(w_K)=\sum \limits_{\substack{\mu \text{ a $v$-excitation}\\ \text{ of } w_K}} M_p (\mu)\cdot t^{\ell(v)}.
\end{equation}
\noindent
\begin{figure}
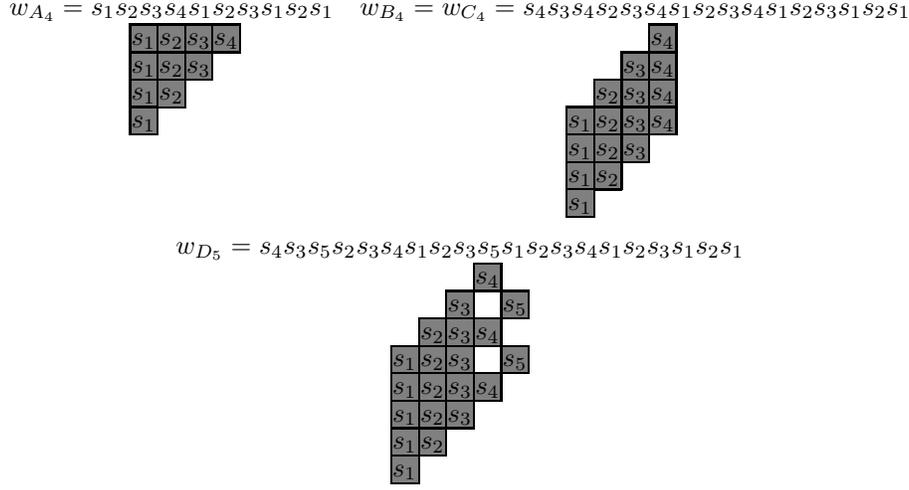
 
$$
\ytableausetup
{mathmode, boxsize=1em}
\begin{array}{cc}
w_{A_4}=s_1s_2s_3s_4s_1s_2s_3s_1s_2s_1
&
w_{B_4}=w_{C_4}=s_4s_3s_4s_2s_3s_4s_1s_2s_3s_4s_1s_2s_3s_1s_2s_1\\
\begin{ytableau}
\none &*(gray)s_1 &*(gray) s_2 &*(gray)s_3  & *(gray) s_4\\
\none &*(gray)s_1 &*(gray) s_2 &*(gray)s_3  & \none \\
\none &*(gray)s_1 & *(gray)s_2 & \none & \none \\
\none & *(gray)s_1 & \none & \none & \none \\
\none &\none & \none & \none & \none \\
\end{ytableau}
&

\begin{ytableau}

\none &\none & \none &  *(gray)s_4 \\
\none &\none & *(gray) s_3& *(gray) s_4 \\
\none &*(gray) s_2& *(gray) s_3& *(gray)s_4  \\
 *(gray)s_1& *(gray)s_2&*(gray) s_3 &*(gray) s_4  \\
 *(gray)s_1& *(gray)s_2&*(gray) s_3 & \none & \none \\
 *(gray)s_1&*(gray) s_2& \none & \none & \none \\
 *(gray)s_1&\none & \none & \none & \none \\
\end{ytableau}
\end{array} $$
$$\begin{array}{c}
w_{D_5}=s_4s_3s_5s_2s_3s_4s_1s_2s_3s_5s_1s_2s_3s_4s_1s_2s_3s_1s_2s_1\\
\begin{ytableau}
\none &\none & \none & *(gray)s_4  \\
\none &\none & *(gray) s_3&  & *(gray) s_5\\
\none &*(gray) s_2&*(gray) s_3 & *(gray)s_4   \\
*(gray) s_1&*(gray)s_2 & *(gray)s_3 &  &*(gray)s_5  \\
 *(gray)s_1&*(gray)s_2 &*(gray) s_3 &*(gray) s_4   \\
 *(gray)s_1& *(gray)s_2& *(gray)s_3 &   \none \\
 *(gray)s_1& *(gray)s_2&   \none & \none \\
 *(gray)s_1& \none & \none & \none \\
\end{ytableau}
\end{array}$$
\caption{Skew diagrams representing the longest element in the classical Lie types.  We use the reduced words for $w_K$ given by Sage.} \label{fig:the word w_K}
\end{figure}
\section{Proof of the Giambelli's formula}
Theorems \ref{thm:by type} and \ref{thm:reduced words} gave two versions of the main theorem of this paper. Having used Theorem \ref{thm:by type} to prove Theorem \ref{thm:reduced words}, we now prove Theorem \ref{thm:by type} case by case according to Lie type.  This section first  gives a proof for type $A$ which will be used for the proofs of the other classical types.  Then we prove the theorem for the exceptional Lie types.  This involves computer-generated proofs for the $E$ series and explicit calculations for types $F_4$ and $G_2$. \\
\\
Throughout this section we assume that $K$ is a connected root system of the given Lie type.  
\subsection{Type A}
\begin{figure}[h]
\begin{tikzpicture}
\node [left] at (-.5,0)  {$A_n:$};
\draw[] (0,0) circle [radius=0.08];
\node [below] at (0,0)  {$s_1$};
\draw[] (1,0) circle [radius=0.08];
\node [below] at (1,0)  {$s_2$};
\draw (0.08,0) -- (.92,0);
\draw[] (2,0) circle [radius=0.08];
\node [below] at (2,0)  {$s_3$};
\draw (1.08,0) -- (1.92,0);
\draw[] (3,0) circle [radius=0.08];
\node [below] at (3,0)  {$s_{n-2}$};
\draw [dashed] (2.08,0) -- (2.92,0);
\draw[] (4,0) circle [radius=0.08];
\node [below] at (4,0)  {$s_{n-1}$};
\draw (3.08,0) -- (3.92,0);
\draw[] (5,0) circle [radius=0.08];
\node [below] at (5,0)  {$s_{n}$};
\draw (4.08,0) -- (4.92,0);
\end{tikzpicture}
\end{figure}

\noindent
While Giambelli's formula for type $A$ was fully addressed by Bayegan and Harada \cite{Giambelli}, we give a proof using modified excited Young diagrams.  The diagrams  $\lambda_{w_K}$ and $\lambda_{p(w_K)}$ are \\
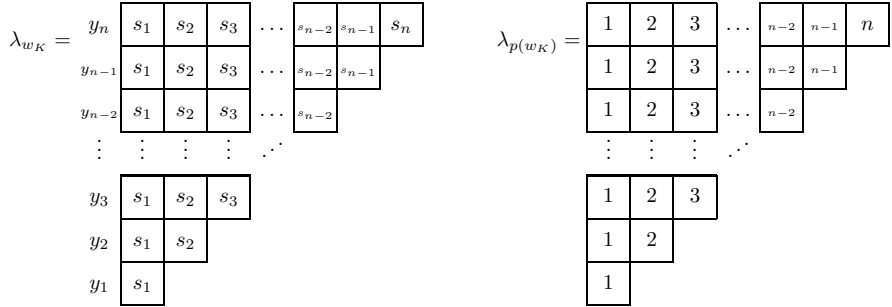
\begin{center}
\scalebox{.8}{
\ytableausetup{boxsize=2em}
$\lambda_{w_K}=$
\begin{ytableau}
\none[y_n]& s_1 & s_2 & s_3 & \none[\dots]
& \scriptscriptstyle s_{n-2}  & \scriptscriptstyle s_{n-1} &  s_n \\
\none[\scriptstyle y_{n-1}]& s_1 & s_2 & s_3 & \none[\dots]
 & \scriptscriptstyle s_{n-2}  & \scriptscriptstyle s_{n-1}  \\
\none[\scriptstyle y_{n-2}]& s_1 & s_2 & s_3 & \none[\dots]
& \scriptscriptstyle s_{n-2}  \\
\none[\vdots]& \none[\vdots] & \none[\vdots]& \none[\vdots] & \none [\iddots] \\
\none[y_3]&  s_1 & s_2 & s_3 \\
\none[y_2]& s_1 & s_2 \\
\none[y_1]& s_1
\end{ytableau}
\hspace{1cm}
$\lambda_{p(w_K)}=$
\begin{ytableau}
1 & 2 & 3 & \none[\dots]
& \scriptscriptstyle{n-2} & \scriptscriptstyle{n-1} &  n \\
1 & 2 & 3 & \none[\dots]
&\scriptscriptstyle{n-2}& \scriptscriptstyle{n-1}\\
1 & 2 & 3 & \none[\dots]
& \scriptscriptstyle{n-2} \\
\none[\vdots] & \none[\vdots]& \none[\vdots] & \none [\iddots] \\
 1 & 2 & 3 \\
1 & 2 \\
1\\
\end{ytableau}}
\end{center}
\noindent
The labels of $\lambda_{p(w_K)}$  can be seen by considering the $i^{th}$ box of the $j^{th}$ row.  By calling the word made by the top row of $\lambda_{w_K}$ $y_n$ and the $j^{th}$ row $y_{n-j+1}$ we see that the root associated with that box is
\begin{align*}
y_ny_{n-1}\cdots y_{n-j+2}s_1s_2\cdots s_{i-1}(\alpha_i)&=y_ny_{n-1}\cdots y_{n-j+2}(\alpha_1+\cdots + \alpha_i)\\
&=y_ny_{n-1}\cdots y_{n-j+3}(\alpha_2+\cdots + \alpha_{i+1})\\
&\vdots \\
&=\alpha_{j}+\cdots +\alpha_{j+i-1}.
\end{align*}
\noindent 
so the label in that box of $\lambda_{p(w_K)}$ is $i$. We use the diagrams to prove the type $A$ case of Theorem \ref{thm:by type}.
\noindent
\begin{Proposition} Theorem \ref{thm:by type} holds when $K$ is a connected type $A$ root system.
\end{Proposition}
\begin{proof} Lemma ~\ref{lemma one product}  showed $C\cdot p_{v_K}=\prod \limits_{1 \leq i \leq n} p_{s_i}$. Evaluating this equation at the fixed point $w_K$ gives
\begin{equation}\label{eq:type a}
C\cdot p_{v_K}(w_K)=\prod \limits_{1 \leq i \leq n} p_{s_i}(w_K).
\end{equation}
There is only one filling $w_K$ with $v_K=s_1s_2\cdots s_n$, specifically
\begin{center}\scalebox{.9}{$\mu=$
\begin{ytableau}
*(gray) s_1 &*(gray) s_2 &*(gray) s_3 & \none[\dots]
&*(gray) \scriptscriptstyle s_{n-2} & *(gray)\scriptscriptstyle s_{n-1} & *(gray) s_n \\
s_1 & s_2 & s_3 & \none[\dots]
& \scriptscriptstyle s_{n-2} & \scriptscriptstyle s_{n-1} \\
s_1 & s_2 & s_3 & \none[\dots]
& \scriptscriptstyle s_{n-2} \\
\none[\vdots] & \none[\vdots]& \none[\vdots] & \none [\iddots] \\
 s_1 & s_2 & s_3 \\
s_1 & s_2 \\
s_1
\end{ytableau}}\end{center}
\noindent
Thus $p_{v_K}(w_K)=M(\mu)\cdot t^n=n!\cdot t^n$. We must also evaluate $p_{s_i}(w_K)$ for each $s_i \in K$.  From $\lambda_{p(w_K)}$ each $s_i$-excitation $\mu$ of $w_K$ has $M(\mu)=i$. From the diagram, there are $n-i+1$ such excitations. So
$$
 p_{s_i}(w_K)=\sum \limits_{\substack{\mu \text{ a $s_i$-excitation}\\  \text{ of } w_K}} M (\mu)\cdot t^{\ell(s_i)}=(n-i+1)i\cdot t.
$$ 
Solving Equation \eqref{eq:type a} for $C$ we obtain 
$$C\cdot n! \cdot t^n=\prod \limits_{1 \leq i \leq n}[(n-i+1)i\cdot t]=\prod \limits_{1 \leq i \leq n}(n-i+1) \cdot \prod \limits_{1 \leq i \leq n}(i\cdot t)=(n!)^2\cdot t^n$$
\noindent
which gives $C=n!=|K|!$ as desired.
   \end{proof}
\noindent
The type $A$ result greatly simplifies the proof for the other classical Lie types.
\begin{Lemma} \label{lemma:reflection s_n}
Let $K$ be a connected root system of type $B_n,C_n,$ or $D_n$.  Then $J=K\setminus \lbrace s_n \rbrace$ is a root subsystem of type $A_{n-1}$ and
\begin{equation}\label{eq:type b setup}
c^K_{n,J} \cdot (n-1)! \cdot p_{v_K}=\prod_{\alpha_i \in K} p_{s_i}
\end{equation} where $c^K_{n,J}=p_{s_n}(w_K)\cdot \frac{p_{v_J}(w_K)}{p_{v_K}(w_K)}.$
\end{Lemma}
\begin{proof}
By the proof of Giambelli's formula for type $A$ and Theorem \ref{thm:Monk formula} respectively, \begin{equation}
(n-1)! \cdot p_{v_J}=\prod_{\alpha_i \in J} p_{s_i}
\hspace{1cm} \text{    
and} \hspace{1cm}
p_{s_n}\cdot p_{v_J}=c^K_{n,J} \cdot p_{v_K}.
\end{equation}
\noindent 
Combining these gives Equation \eqref{eq:type b setup}.  By Theorem \ref{thm:Monk formula} $$c^K_{n,J}=(p_{s_n}(w_K)-p_{s_n}(w_J))\cdot \frac{p_{v_J}(w_K)}{p_{v_K}(w_K)}.$$  By construction the root $\alpha_n$ is not in $J$ so $p_{s_n}(w_J)=0$ giving the desired result.
   \end{proof}

\noindent
Now we will prove Giambelli's formula for the other classical Lie types.

\subsection{Type B}
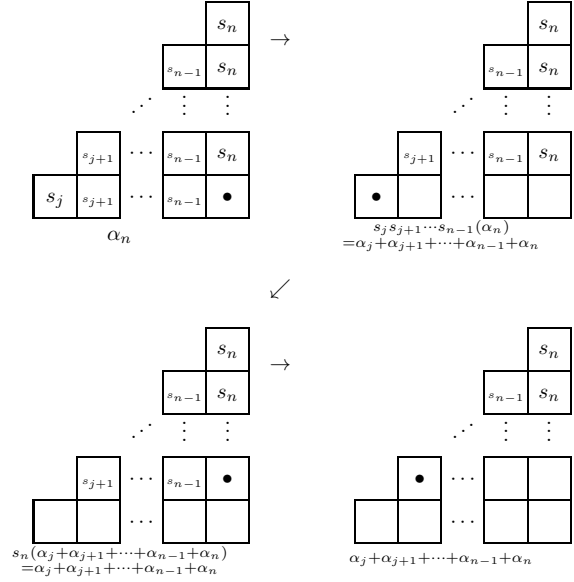
\begin{figure}[h]
\begin{tikzpicture}
\draw[] (0,-2) circle [radius=0.08];
\node [below] at (0,-2)  {$s_1$};
\draw[] (1,-2) circle [radius=0.08];
\node [below] at (1,-2)  {$s_2$};
\draw (0.08,-2) -- (.92,-2);
\draw[] (2,-2) circle [radius=0.08];
\node [below] at (2,-2)  {$s_3$};
\draw (1.08,-2) -- (1.92,-2);
\draw[] (3,-2) circle [radius=0.08];
\node [below] at (3,-2)  {$s_{n-2}$};
\draw [dashed] (2.08,-2) -- (2.92,-2);
\draw[] (4,-2) circle [radius=0.08];
\node [below] at (4,-2)  {$s_{n-1}$};
\draw (3.08,-2) -- (3.92,-2);
\draw[] (5,-2) circle [radius=0.08];
\node [below] at (5,-2)  {$s_{n}$};
\draw (4.05,-1.95) -- (4.95,-1.95);
\draw (4.05,-2.05) -- (4.95,-2.05);
\draw (4.45,-1.9) -- (4.55,-2) -- (4.45,-2.1);
\node at (-1,-2){$B_n$:};
\end{tikzpicture}
\end{figure}

\begin{Proposition}
Theorem \ref{thm:by type}  holds when $K$ is a connected type $B$ root system.
\end{Proposition}
\begin{proof}
Let $K$ be a connected type $B_n$ root system and $J\subset K= K\setminus \{ s_n \}$. By Lemma \ref{lemma:reflection s_n} showing that $c^K_{n,J}=n$ is sufficient to prove Giambelli's formula for type $B$.\\
\\
If $K$ is of type $B_n$ the diagram of the chosen reduced word for $w_K$ is given below.  Each row is labeled by the word of reflections in that row.  For example $x_2=s_2s_3\cdots s_{n-1}s_n$.\begin{center}
\scalebox{.8}{
\ytableausetup{boxsize=2em}
$$\begin{ytableau}
\none\\
\none\\
\none\\
\none\\
\none\\
\none\\
\none [\lambda_{w_K}=]&\none &\none \\
\end{ytableau}
\begin{ytableau}
\none [x_n] &\none&\none &\none &\none &\none & \none & \none & *(gray) s_n\\
\none [x_{n-1}] &\none&\none &\none &\none &\none & \none &   \scriptscriptstyle s_{n-1}&*(gray)  s_n  \\
\none [x_{n-2}] &\none&\none &\none & \none &\none &  \scriptscriptstyle s_{n-2}&  \scriptscriptstyle s_{n-1}& *(gray)  s_n \\
\none [\vdots] &\none &\none & \none &\none &\none[\iddots] & \none [\vdots] & \none [\vdots] & \none [\vdots] \\
\none [x_3] &\none &\none & \none &  s_3 & \none [\cdots] &   \scriptscriptstyle s_{n-2} &  \scriptscriptstyle s_{n-1}& *(gray)  s_n \\
\none [x_2] &\none&\none &  s_2&  s_3 & \none [\cdots] &   \scriptscriptstyle s_{n-2} &  \scriptscriptstyle s_{n-1} & *(gray)  s_n \\
\none [x_1] &\none& *(gray)  s_1& *(gray) s_2&*(gray)  s_3 & \none [\cdots] & *(gray)  \scriptscriptstyle s_{n-2} & *(gray)\scriptscriptstyle s_{n-1} &*(gray)s_n \\
\none [y_{n-1}] &\none & *(gray)  s_1& *(gray) s_2&*(gray)  s_3 & \none [\cdots] & *(gray)  \scriptscriptstyle s_{n-2}& *(gray) \scriptscriptstyle s_{n-1}  \\ 
\none [y_{n-2}] &\none &   s_1&  s_2&  s_3 & \none [\cdots] &   \scriptscriptstyle s_{n-2}   \\ 
\none [\vdots] &\none &\none [\vdots]  & \none [\vdots]  & \none [\vdots] &\none [\iddots] \\
\none [y_3] &\none & s_1& s_2& s_3 & \none & \none \\
\none [y_2] &\none & s_1& s_2& \none & \none & \none \\
\none [y_1] &\none & s_1&\none & \none & \none & \none \\
\end{ytableau}$$
\ytableausetup{boxsize=normal}
}
\end{center}
\noindent
To compute $c_{n,J}^K$ we need to compute $p_v(w_K)$ where $v$ is $s_n$, $s_1s_2\cdots s_{n-1}$, and $s_1s_2\cdots s_{n}$.  All of the $v$-excitations of $w_K$ for these words are contained in the shaded area of the $\lambda_{w_K}$ above.  So we only need the entries of $\lambda_{p(w_K)}$ in those shaded boxes.  Start with the box labeled $s_n$ in row $x_j$ of the diagram.  The reflections that come after do not act on the root, so we look at $x_n x_{n-1} \cdots x_j$ and calculate the root as it moves through the diagram. A bullet, $\bullet$, marks the location of the root in the diagram at each step.  The initial root is below the first diagram and we follow how it changes.
\begin{center}
\scalebox{.8}{
\ytableausetup{boxsize=2em}
$\begin{matrix}
\begin{ytableau}
\none &\none &\none & \none & \none &  s_n\\
\none &\none &\none & \none &   \scriptscriptstyle s_{n-1}&  s_n  \\
\none & \none &\none &\none[\iddots]  & \none [\vdots] & \none [\vdots] \\
\none & \none &  \scriptscriptstyle s_{j+1} & \none [\cdots] &    \scriptscriptstyle s_{n-1}&   s_n \\
\none &  s_j&  \scriptscriptstyle s_{j+1} & \none [\cdots] &    \scriptscriptstyle s_{n-1} & \bullet \\
\end{ytableau}
& \rightarrow &
\begin{ytableau}
\none &\none &\none & \none & \none &  s_n\\
 \none &\none &\none & \none &   \scriptscriptstyle s_{n-1}&  s_n  \\
\none & \none &\none &\none[\iddots]  & \none [\vdots] & \none [\vdots] \\
\none & \none &  \scriptscriptstyle s_{j+1} & \none [\cdots] &    \scriptscriptstyle s_{n-1}&   s_n \\
\none &  \bullet &   & \none [\cdots] &     &  \\
\end{ytableau}\\
\alpha_n &  & \substack{ \\ s_js_{j+1}\cdots s_{n-1}(\alpha_n) \\ =\alpha_j+\alpha_{j+1}+\cdots + \alpha_{n-1} +\alpha_n} \\
\\
& \swarrow & \\
\\
\begin{ytableau}
\none &\none &\none & \none & \none &  s_n\\
\none &\none &\none & \none &   \scriptscriptstyle s_{n-1}&  s_n  \\
\none & \none &\none &\none[\iddots]  & \none [\vdots] & \none [\vdots] \\
\none & \none &  \scriptscriptstyle s_{j+1} & \none [\cdots] &    \scriptscriptstyle s_{n-1}&   \bullet \\
\none &   &   & \none [\cdots] &     &  \\
\end{ytableau}
& \rightarrow &
\begin{ytableau}
\none &\none &\none & \none & \none &  s_n\\
\none &\none &\none & \none &   \scriptscriptstyle s_{n-1}&  s_n  \\
\none & \none &\none &\none[\iddots]  & \none [\vdots] & \none [\vdots] \\
\none & \none &  \bullet & \none [\cdots] &    &    \\
\none &   &   & \none [\cdots] &     &  \\
\end{ytableau}
\\
 \substack{\\ s_n( \alpha_j+\alpha_{j+1}+\cdots + \alpha_{n-1} + \alpha_n)\\ = \alpha_j+\alpha_{j+1}+\cdots + \alpha_{n-1} + \alpha_n}& & \substack{ \alpha_{j}+ \alpha_{j+1} +\cdots + \alpha_{n-1} + \alpha_n}\\
\end{matrix}
$
\ytableausetup{boxsize=normal}}
\end{center}
By the time the bullet gets to the position in the lower left, the root is $\alpha_j+\cdots \alpha_n$ which is invariant under all simple reflections except $s_j$ and $s_{j-1}$.  Neither of those reflections act on the bullet as it continues through the diagram.  Thus the label on the box $s_n$ in row $x_j$ of $\lambda_{p(w_K)}$ is $n-j+1$.\\
\\
We can start the bullet in any box of the diagram.  Suppose that the $h^{th}$ simple reflection of $w_k$ is the $i^{th}$ box in row $y_{n-1}$. Then $\mathbf{r}(\mathbf{h},w_K)$ will be 
$$x_n\cdots x_1 s_1s_2 \cdots s_{i-1}s_{i-1} (\alpha_i)= x_n\cdots x_1  \left(\sum \limits_{m=1}^i \alpha_m \right).$$
\noindent If $i<n-1$ then applying the row $x_1$ to $\sum \limits_{m=1}^i \alpha_m $ gives  $\sum \limits_{m=1}^i \alpha_{m+1} $. Every subsequent $x_j$ again increases each index by one until applying $x_{n-i-1}$ gives  $\sum \limits_{m=n-i-1}^{n-1} \alpha_m $.  If $i=1$ then we are done. Otherwise the action of $s_n$ on this sum results in $\left[ \sum \limits_{m=n-i-1}^{n-1} \alpha_m \right]+2\alpha_n$. each subsequent reflection in $x_{n-i}$ results in the double appearance of another root.
$$x_{n-i} \left( \sum \limits_{m=n-i-1}^{n-1} \alpha_m \right)= \alpha_{n-i-1}+ 2\sum \limits_{m=n-i}^{n} \alpha_m$$
 This root is invariant under the action of every reflection in $x_{n-i+1}$ until the last one, $s_{n-i+1}$ where
$$s_{n-i+1}\left( \alpha_{n-i-1}+ 2\sum \limits_{m=n-i}^{n} \alpha_m\right)=\alpha_{n-i-1}+ \alpha_{n-i}+ 2\sum \limits_{m=n-i+1}^{n} \alpha_m$$
Subsequently 
$$x_{n-i+1}\left(\alpha_{n-i-1}+ \alpha_{n-i}+ 2\sum \limits_{m=n-i+1}^{n} \alpha_m\right) =\alpha_{n-i-1}+ \alpha_{n-i}+\alpha_{n-i+1}+ 2\sum \limits_{m=n-i+2}^{n} \alpha_m$$
and this process continues until the last reflection,
$$x_n\left(\alpha_{n-i-1}+ \alpha_{n-i}+\cdots +\alpha_{n-1}+ 2 \alpha_n \right)=  \sum \limits_{m=1}^{i} \alpha_{n-m}.$$
Thus the entry in the corresponding box of $\lambda_{p(w_K)}=\frac{1}{t}\pi \left(\sum \limits_{m=1}^{i} \alpha_{n-m} \right)=i$.\\
\\
Let the $h^{th}$ simple reflection of $w_K$ to be the $i^{th}$ box of row $x_1$ for some $i\neq n$. The root $\mathbf{r}(\mathbf{h},w_K)$ is 
$$\begin{array}{rl}
x_nx_{n-1}\cdots x_{2} s_{1}s_{2} \cdots s_{i-1}(\alpha_i)&=x_nx_{n-1}\cdots x_{2} \left(\sum \limits_{m=1}^i \alpha_m \right)\\
&=\alpha_1+\alpha_2+\cdots +\alpha_{n-1}+2 \alpha_n + \sum \limits_{m=2}^{i}\alpha_{n-m+1}.
\end{array}$$
Thus the entry in the corresponding box of $\lambda_{p(w_K)}$ is
$n+i $.  Now we can label the relevant boxes of $\lambda_{p(w_K)}$:
\ytableausetup
{mathmode, boxsize=1.7em}
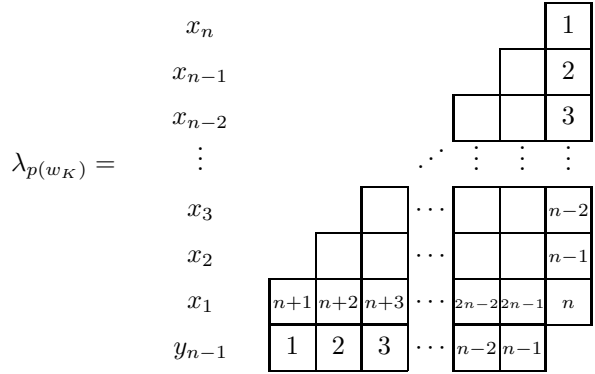
$$\begin{ytableau}
\none\\
\none\\
\none\\
\none [\lambda_{p(w_K)}= ]&\none&\none
\end{ytableau}
\begin{ytableau}
\none [x_n] &\none&\none &\none &\none &\none & \none & \none &  1\\
\none [x_{n-1}] &\none&\none &\none &\none &\none & \none &    &  2  \\
\none [x_{n-2}] &\none&\none &\none & \none &\none & & &   3 \\
\none[\vdots] &\none&\none & \none &\none &\none [\iddots] & \none [\vdots] & \none [\vdots] & \none [\vdots] \\
\none [x_3] &\none&\none & \none & & \none [\cdots] & & &\scriptstyle n-2\\
\none [x_2] &\none&\none & & & \none [\cdots] &  & & \scriptstyle n-1\\
 \none [x_1] &\none& \scriptstyle n+1 & \scriptstyle n+2& \scriptstyle n+3 & \none [\cdots] &  \scriptscriptstyle 2n-2 &  \scriptscriptstyle 2n-1 & \scriptstyle n \\
\none [y_{n-1}] &\none&  1&  2& 3 & \none [\cdots] & \scriptstyle {n-2} & \scriptstyle {n-1}  \\ 
\end{ytableau}$$
With this labeling established, we can see that $p_{s_n}(w_K)=\frac{n(n+1)}{2}t$. We also observe that there is only one $v_K$-excitation of $w_K$ since $v_K$ contains, in order, the reflections $s_1$ through $s_n$ which appear in that order in row $x_1$ and in no other subword of $w_K$.  So $$p_{v_K}(w_K)=(n+1)(n+2) \cdots  (2n-2)(2n-1)n\cdot t^n=\frac{(2n-1)!}{(n-1)!}t^n.$$
The last piece is to calculate $p_{v_J}(w_K)$.  All subwords of $w_K$ that are reduced words for $v_{J}$ are entirely contained within the $x_1y_{n-1}$ subword of $w_K$.  We look at the excited Young diagrams in just those two rows of $\lambda_{p(w_K)}$. A $v_J$-excitation $\mu$ of these two rows is determined by how many boxes it takes from row $x_1$.  The excitation $\mu$ that uses $i$ boxes from row $x_1$  looks like:
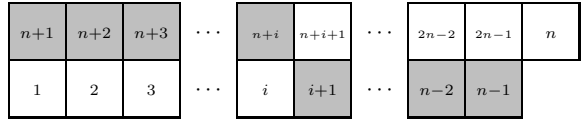
\begin{center}
\scalebox{.85}{
\ytableausetup{boxsize=2.5em}
$\begin{ytableau}
*(lightgray)   \scriptstyle n+1 & *(lightgray) \scriptstyle n+2& *(lightgray) \scriptstyle n+3 &  \none [\cdots] & *(lightgray) \scriptscriptstyle n+i& \scriptscriptstyle n+i+1& \none [\cdots] &  \scriptscriptstyle 2n-2 &  \scriptscriptstyle 2n-1 & \scriptstyle n \\
\scriptstyle 1& \scriptstyle 2&\scriptstyle 3 & \none [\cdots]  & \scriptstyle  i &*(lightgray) \scriptstyle  i+1 & \none [\cdots] & *(lightgray) \scriptstyle {n-2} &*(lightgray) \scriptstyle {n-1}  \\ 
\end{ytableau}$
\ytableausetup{boxsize=normal}
}
\end{center}
This $v_J$-excitation contributes coefficient $M(\mu)=\frac{(n+i)!}{n\cdot i!}$ to $p_{v_J}(w_K)$.  Since we are restricted to  $0\leq i\leq n-1$, 
$$p_{v_J}(w_K)=\sum \limits_{i=0}^{n-1} \frac{(n+i)!}{n\cdot i!} t^{n-1}.$$
Putting all of the pieces together we have
$$\begin{array}{rl}
c_{n,J}^{K}&=\frac{p_{s_n}(w_K)\cdot p_{v_{J}}(w_K)}{p_{v_K}(w_K)}\\
\\
&=\frac{\frac{n(n+1)}{2}t}{\frac{(2n-1)!}{(n-1)!}t^n} \cdot \sum \limits_{i=0}^{n-1} \frac{(n+i)!}{n\cdot i!}t^{n-1}\\
\\
&=\frac{(n+1)!}{2n \cdot (2n-1)!} \cdot \sum \limits_{i=0}^{n-1} \frac{(n+i)!}{i!}.
\end{array}$$
By combinatorial identity, the sum can be rewritten and
$$c_{n,J}^K=\frac{(n+1)!}{2n \cdot (2n-1)!} \cdot \frac{n\cdot (2n)!}{(n+1)!}=n.$$
   \end{proof}

\subsection{Type C}
\begin{figure}[h]
\begin{tikzpicture}
\draw[] (0,-4) circle [radius=0.08];
\node [below] at (0,-4)  {$s_1$};
\draw[] (1,-4) circle [radius=0.08];
\node [below] at (1,-4)  {$s_2$};
\draw (0.08,-4) -- (.92,-4);
\draw[] (2,-4) circle [radius=0.08];
\node [below] at (2,-4)  {$s_3$};
\draw (1.08,-4) -- (1.92,-4);
\draw[] (3,-4) circle [radius=0.08];
\node [below] at (3,-4)  {$s_{n-2}$};
\draw [dashed] (2.08,-4) -- (2.92,-4);
\draw[] (4,-4) circle [radius=0.08];
\node [below] at (4,-4)  {$s_{n-1}$};
\draw (3.08,-4) -- (3.92,-4);
\draw[] (5,-4) circle [radius=0.08];
\node [below] at (5,-4)  {$s_{n}$};
\draw (4.05,-3.95) -- (4.95,-3.95);
\draw (4.05,-4.05) -- (4.95,-4.05);
\draw (4.55,-3.9) -- (4.45,-4) -- (4.55,-4.1);
\node at (-1,-4){$C_n$:};
\end{tikzpicture}
\end{figure}
\begin{Proposition}
Theorem \ref{thm:by type} holds when $K$ is a connected type $C$ root system.
\end{Proposition}
\begin{proof}
This proof mirrors the proof in type $B$. Let $K$ be a connected type $C_n$ root system and define $J\subset K$ to be $J=K\setminus \{ s_n \}$. By Lemma \ref{lemma:reflection s_n} showing that $c^K_{n,J}=n$ is sufficient to prove Giambelli's formula for type $C$.\\
\\
The longest word $w_K$ is the same for type $C_n$ as for type $B_n$, the only changes from type $B$ are the box labels of  $\lambda_{p(w_K)}$.  First we find the label for the box corresponding to reflection $s_n$ in row $x_j$.

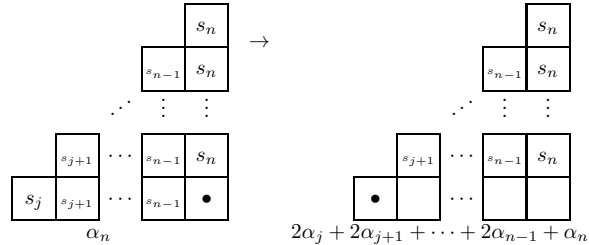
\begin{center}
\scalebox{.8}{
\ytableausetup{boxsize=2em}
$\begin{matrix}
\begin{ytableau}
\none &\none &\none & \none & \none &  s_n\\
\none &\none &\none & \none &   \scriptscriptstyle s_{n-1}&  s_n  \\
\none & \none &\none &\none[\iddots]  & \none [\vdots] & \none [\vdots] \\
\none & \none &  \scriptscriptstyle s_{j+1} & \none [\cdots] &    \scriptscriptstyle s_{n-1}&   s_n \\
\none &  s_j&  \scriptscriptstyle s_{j+1} & \none [\cdots] &    \scriptscriptstyle s_{n-1} & \bullet \\
\end{ytableau}
& \rightarrow &
\begin{ytableau}
\none &\none &\none & \none & \none &  s_n\\
 \none &\none &\none & \none &   \scriptscriptstyle s_{n-1}&  s_n  \\
\none & \none &\none &\none[\iddots]  & \none [\vdots] & \none [\vdots] \\
\none & \none &  \scriptscriptstyle s_{j+1} & \none [\cdots] &    \scriptscriptstyle s_{n-1}&   s_n \\
\none &  \bullet &   & \none [\cdots] &     &  \\
\end{ytableau}
%&
%\rightarrow
%&
%\begin{ytableau}
%\none &\none &\none & \none & \none &  s_n\\
%\none &\none &\none & \none &   \scriptscriptstyle s_{n-1}&  s_n  \\
%\none & \none &\none &\none[\iddots]  & \none [\vdots] & \none [\vdots] \\
%\none & \none &  \bullet & \none [\cdots] &    &    \\
%\none &   &   & \none [\cdots] &     &  \\
%\end{ytableau}
\\
\alpha_n &  & 2\alpha_j+2\alpha_{j+1}+\cdots + 2\alpha_{n-1} +\alpha_n & 
%& 2\alpha_j+2\alpha_{j+1}+\cdots + 2\alpha_{n-1} +\alpha_n
\\
\end{matrix}
$
\ytableausetup{boxsize=normal}
}
\end{center}
\noindent
The root $2\alpha_j+2\alpha_{j+1}+\cdots + 2\alpha_{n-1} +\alpha_n$ is invariant under all reflections except $s_j$ and $s_{j-1}$ which will not act on the root.  Thus the label in $\lambda_{p(w_K)}$ is $2(n-j)+1$. Adding up all the labels of these boxes gives that $$p_{s_n}(w_K)=\sum \limits_{j=1}^n (2j-1)\cdot t =n^2\cdot t.$$
To compute $p_{v_J}(w_K)$ and $p_{v_K}(w_K)$ we need to compute the $\lambda_{p(w_K)}$ diagram labels of rows $y_{n-1}$ and $x_1$. If the $h^{th}$ box is box $i\neq n$ of row $x_1$ then $$\mathbf{r}(\mathbf{h},w_K)=x_nx_{n-1}\cdots x_{2} s_{1}s_{2}\cdots s_{i-1}(\alpha_i)=
x_nx_{n-1}\cdots x_{2} \left(\sum \limits_{m=1}^i \alpha_m\right). $$
By moving a root through the diagram, $x_nx_{n-1}\cdots x_{2}(\alpha_1)=\alpha_1+\alpha_2+\cdots +\alpha_n$ and if $1<i<n$ then $x_nx_{n-1}\cdots x_{2}(\alpha_i)$ gets pushed through the diagram as follows:
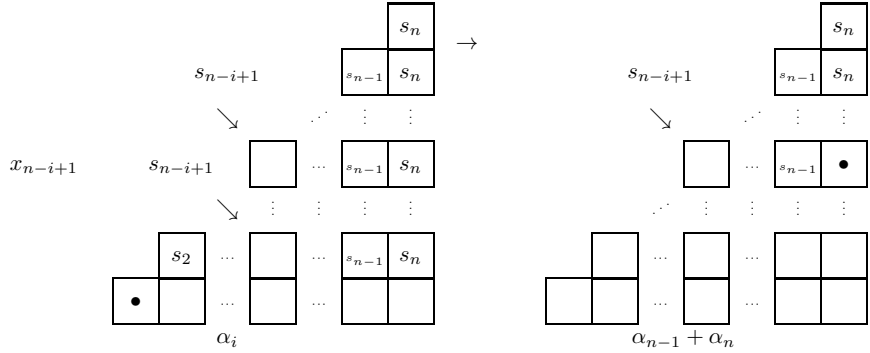
\begin{center}
\scalebox{.85}{
\ytableausetup{boxsize=2em}
$\begin{matrix}
\begin{ytableau}
\none &\none & \none & \none &\none &\none & \none & \none &  s_n\\
\none &\none &\none & \none  &\none [s_{n-i+1}] &\none & \none &   \scriptscriptstyle s_{n-1}&  s_n  \\
\none &\none &\none & \none & \none [\searrow] &\none &\none[\scriptscriptstyle\iddots]  & \none [\scriptscriptstyle\vdots] & \none [\scriptscriptstyle \vdots] \\
\none [x_{n-i+1}] &\none &\none &\none [ s_{n-i+1}] &\none &   & \none[\scriptscriptstyle\cdots] & \scriptscriptstyle s_{n-1}&  s_n  \\
\none &\none & \none &\none &\none[\searrow] &\none[ \scriptscriptstyle \vdots] & \none [\scriptscriptstyle\vdots] & \none [\scriptscriptstyle\vdots]& \none [\scriptscriptstyle\vdots] \\
\none &\none & \none &  s_{2}  & \none [\scriptscriptstyle\cdots] &   &\none[\scriptscriptstyle\cdots]& \scriptscriptstyle s_{n-1} & s_n  \\
\none &\none & \bullet &   & \none [\scriptscriptstyle\cdots] &   &\none[\scriptscriptstyle\cdots]&      &  \\
\end{ytableau}
& \rightarrow &
\begin{ytableau}
\none & \none & \none &\none &\none & \none & \none &  s_n\\
\none &\none & \none &\none  [s_{n-i+1}]&\none & \none &   \scriptscriptstyle s_{n-1}&  s_n  \\
\none &\none & \none & \none [\searrow]&\none &\none[\scriptscriptstyle\iddots]  & \none [\scriptscriptstyle\vdots] & \none [\scriptscriptstyle \vdots] \\
\none &\none &\none &\none &   & \none[\scriptscriptstyle\cdots] & \scriptscriptstyle s_{n-1}&   \bullet \\
\none & \none &\none &\none[\scriptscriptstyle\iddots] &\none[ \scriptscriptstyle \vdots] & \none [\scriptscriptstyle\vdots] & \none [\scriptscriptstyle\vdots]& \none [\scriptscriptstyle\vdots] \\
\none & \none &   & \none [\scriptscriptstyle\cdots] &   &\none[\scriptscriptstyle\cdots]&  &   \\
\none &  &   & \none [\scriptscriptstyle\cdots] &   &\none[\scriptscriptstyle\cdots]&      &  \\
\end{ytableau} \\
\alpha_i &  & \alpha_{n-1} + \alpha_{n} \\
\end{matrix}
$

\ytableausetup{boxsize=normal}
}
\end{center}
Since we are working in type $C$ the reflection $s_{n-1}$ sends $\alpha_n$ to $2\alpha_{n-1}+\alpha_n$ so the next row of the diagram acts like this:
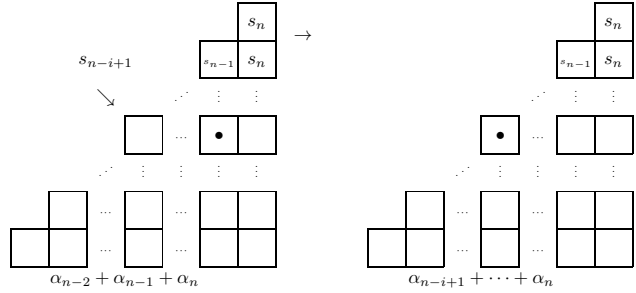
\begin{center}
\scalebox{.7}{
\ytableausetup{boxsize=2em}
$
\begin{matrix}
\begin{ytableau}
\none & \none & \none &\none &\none & \none & \none &  s_n\\
\none &\none & \none &\none  [s_{n-i+1}] &\none & \none &   \scriptscriptstyle s_{n-1}&  s_n  \\
\none &\none & \none & \none[\searrow] &\none &\none[\scriptscriptstyle\iddots]  & \none [\scriptscriptstyle\vdots] & \none [\scriptscriptstyle \vdots] \\
\none  &\none &\none &\none &   & \none[\scriptscriptstyle\cdots] & \bullet&   \\
\none & \none &\none &\none[\scriptscriptstyle\iddots] &\none[ \scriptscriptstyle \vdots] & \none [\scriptscriptstyle\vdots] & \none [\scriptscriptstyle\vdots]& \none [\scriptscriptstyle\vdots] \\
\none & \none &   & \none [\scriptscriptstyle\cdots] &   &\none[\scriptscriptstyle\cdots]&  &   \\
\none &  &   & \none [\scriptscriptstyle\cdots] &   &\none[\scriptscriptstyle\cdots]&      &  \\
\end{ytableau}
& \rightarrow &
\begin{ytableau}
\none & \none & \none &\none &\none & \none & \none &  s_n\\
\none &\none & \none &\none &\none & \none &   \scriptscriptstyle s_{n-1}&  s_n  \\
\none &\none & \none & \none &\none &\none[\scriptscriptstyle\iddots]  & \none [\scriptscriptstyle\vdots] & \none [\scriptscriptstyle \vdots] \\
\none &\none &\none &\none &  \bullet & \none[\scriptscriptstyle\cdots] & &    \\
\none & \none &\none &\none[\scriptscriptstyle\iddots] &\none[ \scriptscriptstyle \vdots] & \none [\scriptscriptstyle\vdots] & \none [\scriptscriptstyle\vdots]& \none [\scriptscriptstyle\vdots] \\
\none & \none &   & \none [\scriptscriptstyle\cdots] &   &\none[\scriptscriptstyle\cdots]&  &   \\
\none &  &   & \none [\scriptscriptstyle\cdots] &   &\none[\scriptscriptstyle\cdots]&      &  \\
\end{ytableau} \\
\alpha_{n-2} +\alpha_{n-1} +\alpha_n & & \alpha_{n-i+1} +\cdots + \alpha_n
\end{matrix}$
\ytableausetup{boxsize=normal}
}
\end{center}
\noindent
Row $x_{n-i+2}$ eliminates everything except  $\alpha_{n-i+1}$ which is preserved for the rest of the diagram.  So
$$\mathbf{r}(\mathbf{h},w_K)=x_nx_{n-1}\cdots x_{2} \left(\sum \limits_{m=1}^i \alpha_m\right)= \sum \limits_{m=1}^n \alpha_m +\sum \limits_{m=2}^i \alpha_{n-m+1} $$
and the entry in $\lambda_{p(w_K)}$ is $n+i-1$.  Like in type $B$,  $x_nx_{n-1}\cdots x_2x_1(\alpha_i)= \alpha_{n-i}$ for $i\neq n$.  We can fill in the entries of rows $x_1$ and $y_{n-1}$ of $\lambda_{p(w_K)}$ as follows:
\begin{center}
\scalebox{.85}{
\ytableausetup
{mathmode, boxsize=2.5em}
$\begin{ytableau}
\none [x_1] & *(lightgray)  \scriptstyle n & *(lightgray) \scriptstyle n+1& *(lightgray) \scriptstyle n+2 & \none [\cdots] &  *(lightgray) \scriptscriptstyle n+i-1 & \scriptstyle n+i & \none [\cdots] & \scriptstyle 2n-2 &  \scriptstyle 2n-1 & \scriptstyle n \\
\none [y_{n-1}] & \scriptstyle 1& \scriptstyle 2&\scriptstyle 3 & \none [\cdots] & \scriptstyle i & *(lightgray)\scriptstyle i+1& \none [\cdots]& *(lightgray) \scriptstyle {n-2} & *(lightgray) \scriptstyle {n-1}  \\ 
\end{ytableau}$}
\end{center}
A $v_J$-excitation of $w_K$ is marked in light gray.  Summing over the number of boxes of $\mu$ that are in row $x_1$ as in the previous section gives $$p_{v_J}(w_K)=\sum\limits_{i=0}^{n-1} \frac{(n+i-1)!}{i!}t^{n-1}=\frac{(2n-1)!}{n!}t^{n-1}.$$
We also see that there is only one $v_K$-excitation of $w_K$  so $p_{v_K}(w_K)=\frac{(2n-1)!}{(n-1)!}t^n$. 

\noindent
Putting all the pieces together we obtain $$c_{n,J}^K=p_{s_n}(w_K)\frac{p_{v_{J}}(w_K)}{p_{v_K}(w_K)}= n^2\frac{\frac{(2n-1)!}{n!}}{\frac{(2n-1)!}{(n-1)!}}=\frac{n^2}{n}=n.$$
   \end{proof}

\subsection{Type D}
\begin{figure}[h]
\begin{tikzpicture}
\draw[] (0,-6) circle [radius=0.08];
\node [below] at (0,-6)  {$s_1$};
\draw[] (1,-6) circle [radius=0.08];
\node [below] at (1,-6)  {$s_2$};
\draw (0.08,-6) -- (.92,-6);
\draw[] (2,-6) circle [radius=0.08];
\node [below] at (2,-6)  {$s_3$};
\draw (1.08,-6) -- (1.92,-6);
\draw[] (3,-6) circle [radius=0.08];
\node [below] at (3,-6)  {$s_{n-3}$};
\draw [dashed] (2.08,-6) -- (2.92,-6);
\draw[] (4,-6) circle [radius=0.08];
\node [below] at (4,-6)  {$s_{n-2}$};
\draw (3.08,-6) -- (3.92,-6);
\draw[] (5,-6) circle [radius=0.08];
\node [below] at (5,-6)  {$s_{n-1}$};
\draw (4.08,-6) -- (4.92,-6);
\draw[] (4,-5) circle [radius=0.08];
\node [above] at (4,-5)  {$s_n$};
\draw (4,-5.92) -- (4,-5.08);
\node at (-1,-6){$D_n$};
\end{tikzpicture}
\end{figure}
\begin{Proposition}
Theorem \ref{thm:by type} holds when $K$ is a connected type $D$ root system.
\end{Proposition}
\begin{proof}
Let $K$ be a connected type $D_n$ root system and $J\subset K= K\setminus \{ s_n \}$. By Lemma \ref{lemma:reflection s_n} it suffices to show that $c^K_{n,J}=\frac{n}{2}$.  If $K$ is a root system of type $D_n$ then the shape of $\lambda_{w_K}$ depends on whether $n$ is even or odd. Figure~\ref{fig:type d} gives the two diagrams for type $D_n$.  In each of these shapes, there is only one $v_K$-excitation of $w_K$. This subword occurs in the rows $x_1$ and $y_{n-1}$ and looks like:
$$
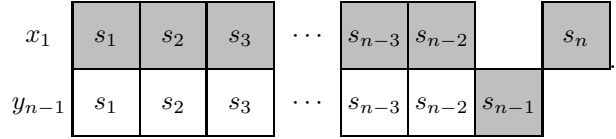
\begin{ytableau}
\none [x_{1}]&*(lightgray) s_1&*(lightgray)s_2 & *(lightgray)s_3 & \none [\cdots]& *(lightgray) s_{n-3}&*(lightgray) s_{n-2}& \none &*(lightgray)s_n  \\
\none [y_{n-1}]& s_1&s_2 & s_3 & \none [\cdots]&s_{n-3}&  s_{n-2}& *(lightgray) s_{n-1}   \\
\end{ytableau}.
$$
\noindent
The $v_J$-excitations of $w_K$ are in the same two rows and there are $n-1$ such excitations. Each excitation $\mu$ looks like:
$$
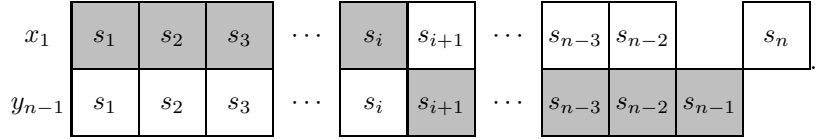
\begin{ytableau}
\none [x_{1}]&*(lightgray) s_1&*(lightgray)s_2 & *(lightgray)s_3  & \none [\cdots]& *(lightgray) s_{i}& s_{i+1}& \none [\cdots]&  s_{n-3}& s_{n-2}& \none &s_n  \\
\none [y_{n-1}]& s_1&s_2 & s_3  & \none [\cdots]& s_{i}&*(lightgray)  s_{i+1}& \none [\cdots]&*(lightgray) s_{n-3}&*(lightgray)  s_{n-2}& *(lightgray) s_{n-1}   \\
\end{ytableau}.
$$

\begin{figure}
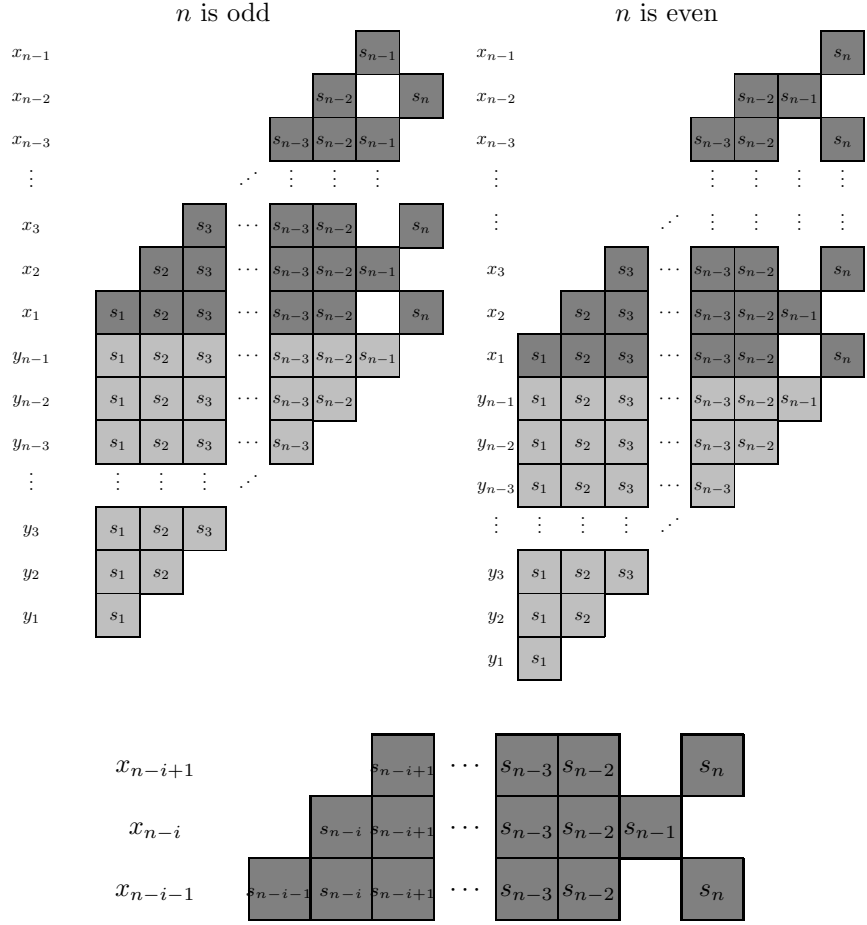
 

\caption{The diagrams of $\lambda_{w_K}$ for $K$ a type $D_n$ root system. } \label{fig:type d}
$$\begin{array}{cc}
n \text{ is odd} & n\text{ is even} \\
\ytableausetup{boxsize=2.3em}
\scalebox{.7}{
\begin{ytableau}
\none[x_{n-1}]&\none &\none &\none & \none & \none &\none &\none & *(gray)s_{n-1}  \\
\none [x_{n-2}]&\none &\none &\none &\none &\none  &\none &  *(gray) s_{n-2} &\none & *(gray) s_{n}\\
\none [x_{n-3}]&\none &\none &\none &\none &\none  &  *(gray) s_{n-3} & *(gray) s_{n-2} & *(gray) s_{n-1}\\
\none [\vdots] &\none &\none &\none &\none &\none  [\iddots] & \none [\vdots] &\none [\vdots] &\none [\vdots]\\
\none [x_{3}]&\none &\none &\none & *(gray) s_3& \none [\cdots]&*(gray) s_{n-3}& *(gray) s_{n-2}& \none  & *(gray) s_n\\
\none [x_{2}]&\none &\none &*(gray) s_2&*(gray) s_3 & \none [\cdots]&*(gray) s_{n-3}& *(gray) s_{n-2}& *(gray)s_{n-1}   \\
\none [x_{1}]&\none &*(gray) s_1&*(gray)s_2 & *(gray)s_3 & \none [\cdots]& *(gray) s_{n-3}&*(gray) s_{n-2}&  &*(gray)s_n  \\
\none [y_{n-1}]&\none & *(lightgray)s_1&*(lightgray)s_2 &*(lightgray) s_3 & \none [\cdots]&*(lightgray) s_{n-3}& *(lightgray) s_{n-2}& *(lightgray) s_{n-1}   \\
\none [y_{n-2}]&\none & *(lightgray)s_1&*(lightgray)s_2 &*(lightgray) s_3 & \none [\cdots]&*(lightgray) s_{n-3}& *(lightgray) s_{n-2}\\
\none [y_{n-3}]&\none & *(lightgray)s_1&*(lightgray)s_2 &*(lightgray) s_3 & \none [\cdots]&*(lightgray) s_{n-3}\\
\none [\vdots] &\none & \none [\vdots] &  \none [\vdots]&  \none [\vdots] & \none [\iddots] & \none \\
\none [y_{3}]&\none & *(lightgray)s_1& *(lightgray)s_2& *(lightgray)s_3 &   \none \\
\none [y_{2}]&\none & *(lightgray)s_1& *(lightgray)s_2&   \none & \none \\
\none [y_{1}]&\none & *(lightgray)s_1& \none & \none & \none \\
\end{ytableau}}
&
\scalebox{.7}{
\begin{ytableau}
\none [x_{n-1}]&\none &\none & \none & \none &\none &\none &\none & *(gray)s_{n}  \\
\none [x_{n-2}]&\none &\none &\none &\none  &\none &  *(gray) s_{n-2}  & *(gray) s_{n-1}\\
\none [x_{n-3}]&\none &\none &\none &\none  &  *(gray) s_{n-3} & *(gray) s_{n-2} & \none & *(gray) s_{n}\\
\none [\vdots] &\none &\none &\none &\none  & \none [\vdots] &\none [\vdots] &\none [\vdots]&\none [\vdots]\\
\none [\vdots] &\none &\none &\none &\none [\iddots] & \none [\vdots] &\none [\vdots] &\none [\vdots]&\none [\vdots]\\
\none [x_3]&\none &\none & *(gray) s_3& \none [\cdots]&*(gray) s_{n-3}& *(gray) s_{n-2}& \none  & *(gray) s_n\\
\none [x_2]&\none &*(gray) s_2&*(gray) s_3 & \none [\cdots]&*(gray) s_{n-3}& *(gray) s_{n-2}& *(gray)s_{n-1}   \\
\none [x_1]&*(gray) s_1&*(gray)s_2 & *(gray)s_3 & \none [\cdots]& *(gray) s_{n-3}&*(gray) s_{n-2}&  &*(gray)s_n  \\
\none  [y_{n-1}]& *(lightgray)s_1&*(lightgray)s_2 &*(lightgray) s_3 & \none [\cdots]&*(lightgray) s_{n-3}& *(lightgray) s_{n-2}& *(lightgray) s_{n-1}   \\
\none [y_{n-2}]& *(lightgray)s_1&*(lightgray)s_2 &*(lightgray) s_3 & \none [\cdots]&*(lightgray) s_{n-3}& *(lightgray) s_{n-2}\\
\none [y_{n-3}]& *(lightgray)s_1&*(lightgray)s_2 &*(lightgray) s_3 & \none [\cdots]&*(lightgray) s_{n-3}\\
\none [\vdots]& \none [\vdots] &  \none [\vdots] &  \none [\vdots]& \none [\iddots] & \none \\
\none [y_{3}]& *(lightgray)s_1& *(lightgray)s_2& *(lightgray)s_3 &   \none \\
\none [y_{2}]& *(lightgray)s_1& *(lightgray)s_2&   \none & \none \\
\none [y_{1}]& *(lightgray)s_1& \none & \none & \none \\
\end{ytableau}}
\\
\end{array}$$
\end{figure}\noindent
We need to find the labels of these boxes in $\lambda_{p(w_K)}$ in order to compute $p_{v_K}(w_K)$ and $p_{v_J}(w_K)$.  Denote by $x$ the word obtained from the first $n-1$ rows of $w_K$, i.e. $x=x_{n-1}x_{n-2}\cdots x_2x_1$. We compute $x(\alpha_i)$ for $ i < n$.\\
\\
First we examine the action of $x$ on $\alpha_i$ for $1<i \leq n-2$.
Suppose that $i<n-2$. Then we take the action of $x$ row by row to get to the root $\alpha_{n-2}$. The first reflection in $x_1$ to not preserve $\alpha_i$ is $s_{i+1}$ which sends it to $\alpha_i+\alpha_{i+1}$.  The next reflection, $s_i$ then brings the root to $\alpha_{i+1}$ which the rest of the reflections in $x_1$ preserves.  Similarly if $x_2(\alpha_{i+1})=\alpha_{i+2}$ if $i+1<n-2$.  This pattern continues until
$$x_{n-i-2}x_{n-i-3}\cdots x_1(\alpha_i)=\alpha_{n-2}$$  
The action of the next three rows, $x_{n-i-1},x_{n-i},$ and $x_{n-i+1}$ depend on whether $n-i$ is even or odd. If $n-i$ is odd then the next three rows have form\\
\begin{center}
\begin{ytableau}
\none [x_{n-i+1}]&\none &\none &\none & *(gray) \scriptstyle s_{n-i+1}& \none [\cdots]&*(gray) s_{n-3}& *(gray) s_{n-2}& \none  & *(gray) s_n\\
\none [x_{n-i}]&\none &\none &*(gray) \scriptstyle s_{n-i}&*(gray) \scriptstyle s_{n-i+1} & \none [\cdots]&*(gray) s_{n-3}& *(gray) s_{n-2}& *(gray)s_{n-1}   \\
\none [x_{n-i-1}]&\none &*(gray) \scriptstyle s_{n-i-1} &*(gray) \scriptstyle s_{n-i}& *(gray) \scriptstyle s_{n-i+1} & \none [\cdots]& *(gray) s_{n-3}&*(gray) s_{n-2}&\none  &*(gray)s_n  \\
\end{ytableau}
\end{center}
The actions of these rows are:
$$ 
\begin{array}{rcl}
x_{n-i-1}(\alpha_{n-2})&=&\alpha_{n-1}\\
x_{n-i}(\alpha_{n-1})&=& \alpha_{n-i}+\cdots +\alpha_{n-2}+\alpha_{n-1}\\
x_{n-i+1}(\alpha_{n-i}+\cdots +\alpha_{n-2}+\alpha_{n-1})&=&\alpha_{n-i}
\end{array}
$$
\noindent
If instead $n-i$ is even the next three rows look like
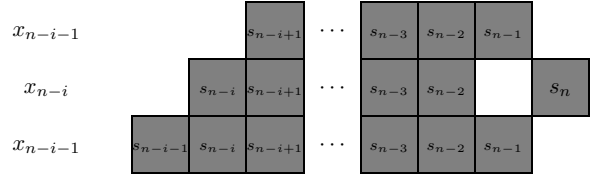
\begin{center}
\scalebox{.85}{
\ytableausetup{boxsize=2.5em}
\begin{ytableau}
\none [x_{n-i-1}]&\none &\none&\none & *(gray)\scriptstyle s_{n-i+1}& \none [\cdots]& *(gray)\scriptstyle s_{n-3}&*(gray)\scriptstyle s_{n-2}&  *(gray) \scriptstyle s_{n-1} & \none \\
\none [x_{n-i}]&\none &\none &*(gray)\scriptstyle s_{n-i} & *(gray)\scriptstyle s_{n-i+1}& \none [\cdots]&*(gray)\scriptstyle s_{n-3}& *(gray) \scriptstyle s_{n-2}& \none  & *(gray) s_n\\
\none [x_{n-i-1}]&\none &*(gray)\scriptstyle s_{n-i-1} &*(gray) \scriptstyle s_{n-i}&*(gray)\scriptstyle s_{n-i+1} & \none [\cdots]&*(gray)\scriptstyle s_{n-3}& *(gray) \scriptstyle s_{n-2}& *(gray)\scriptstyle s_{n-1}   \\
\end{ytableau}}
\end{center}
\noindent
and act by:
$$ 
\begin{array}{rcl}
x_{n-i-1}(\alpha_{n-2})&=&\alpha_{n}\\
x_{n-i}(\alpha_{n-1})&=& \alpha_{n-i}+\cdots +\alpha_{n-2}+\alpha_{n}\\
x_{n-i+1}(\alpha_{n-i}+\cdots +\alpha_{n-2}+\alpha_{n})&=&\alpha_{n-i}
\end{array}
$$
Whether $n-i$ is odd or even, the root $\alpha_{n-i}$ is invariant under the action of $s_j$ for $j>n-i+1$ so $x(\alpha_i)=\alpha_{n-i}$ for all $i$ greater than $1$ and less than $n-2$.\\
\\
If we start with the root $\alpha_1$ then 
$x_{n-3}x_{n-4}\cdots x_{1}(\alpha_1)=\alpha_{n-2}$. The rest of the computation is
$$x_{n-1}x_{n-2}(\alpha_{n-2})=\begin{cases}
s_{n-1}s_{n-2}s_n(\alpha_{n-2})=\alpha_n & \text{ if $n$ is odd}\\
s_{n}s_{n-2}s_{n-1}(\alpha_{n-2})=\alpha_{n-1} &\text{ if $n$ is even}
\end{cases}
$$
\\
Next we address $x(\alpha_{n-1})$. Going row by row, $$x_1(\alpha_{n-1})= \alpha_1+\alpha_2+\cdots +\alpha_{n-2}+\alpha_{n-1} \text{ and }$$
$$ x_2(\alpha_1+\alpha_2+\cdots ++\alpha_{n-2}+\alpha_{n-1})=\alpha_1$$ which is invariant under $s_j$ for $j>2$.  Thus $x(\alpha_{n-1})=\alpha_1$.\\
\\
The result of these computations is that for $i<n$ the word $x$ takes each root $\alpha_i$ to a root $\alpha_j$ and therefor $\pi_1(x(\alpha_i))=t$ for all $i<n$. Since the label in the $i^{th}$ box of row $y_{n-1}$ is the height of the root $xs_1s_2 \cdots s_{i-1}(\alpha_i)=x(\alpha_1+\cdots + \alpha_i)$ that box in $\lambda_{p(w_K)}$ is labeled $i$.\\
\\
We also want to find the root corresponding to the $i^{th}$ box of row $x_1$. A non-reduced way to write the word preceding that box in $w_K$ is $xs_ns_{n-2}s_{n-3}\cdots s_{i+1}s_i$ and the corresponding root is
$$\begin{array}{rl}
xs_n s_{n-2} s_{n-3} \cdots s_{i+1}s_i(\alpha_i)&=x(-\alpha_i-\alpha_{i+1}-\cdots -\alpha_{n-2}-\alpha_n)\\
\\ & =-x(\alpha_i+\alpha_{i+1}+\cdots +\alpha_{n-2}) -x(\alpha_n).\end{array}$$
We compute the action of $x$ on the root $\alpha_n$:
$$x_1(\alpha_n)=-(\alpha_1+\alpha_2+\cdots +\alpha_{n-2}+\alpha_n)$$
$$s_{n-1}(-(\alpha_1+\alpha_2+\cdots +\alpha_{n-2}+\alpha_n))=-(\alpha_1+\alpha_2+\cdots +\alpha_{n-2}+\alpha_{n-1}+\alpha_n)$$
$$s_{n-2}(-(\alpha_1+\alpha_2+\cdots +\alpha_{n-2}+\alpha_{n-1}+\alpha_n))=-(\alpha_1+\alpha_2+\cdots +2\alpha_{n-2}+\alpha_{n-1}+\alpha_n)$$
Each subsequent reflection places the coefficient $2$ in front of another simple root until
$$x_2x_1(\alpha_n)=-(\alpha_1+2\alpha_2+\cdots +2\alpha_{n-2}+\alpha_{n-1}+\alpha_n).$$
This root is invariant under the action of $s_i$ for $i>2$ and thus $$x(\alpha_n)=-(\alpha_1+2\alpha_2+\cdots +2\alpha_{n-2}+\alpha_{n-1}+\alpha_n)$$
\noindent
Thus the root associated with the $i^{th}$ box of row $x_1$ is   
$
xs_n s_{n-2} s_{n-3} \cdots s_{i+1}(\alpha_i)=$
$$-x(\alpha_i+\alpha_{i+1}+\cdots \alpha_{n-2})+ \alpha_1+2\alpha_2+\cdots +2\alpha_{n-2}+\alpha_{n-1}+\alpha_n.
$$
which has height in the root poset $2n-3-(n-i-1)=n+i-2$. We can now label the rows $x_1$ and $y_{n-1}$ in the diagram $\lambda_{p(w_K)}$.
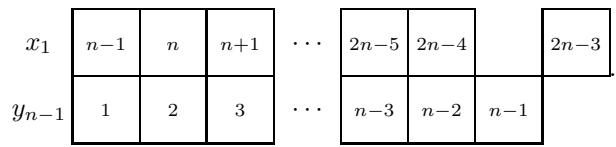
$$
\begin{ytableau}
\none [x_{1}]&\scriptstyle n-1 &\scriptstyle n &\scriptstyle  n+1 & \none [\cdots]& \scriptstyle 2n -5& \scriptstyle2n-4& \none &\scriptstyle2n-3  \\
\none [y_{n-1}]&\scriptstyle 1&\scriptstyle2 &\scriptstyle 3  &  \none [\cdots]& \scriptstyle{n-3}& \scriptstyle {n-2}& \scriptstyle {n-1}   \\
\end{ytableau}.
$$
This means that $p_{v_K}(w_K)=\frac{(2n-3)!}{(n-2)!}(n-1) t^n$. Furthermore when $\mu$ is a $v_J$-excitation with $i$ boxes in the top row contributes $M(\mu)=(n-1)\frac{(n+i-2)!}{i!}$ to $p_{v_J}(w_K)$.   Therefore
$$p_{v_J}(w_K)=\left[ \sum \limits_{i=0}^{n-2} (n-1)\frac{(n+i-2)!}{i!} \right] t^{n-1} = \frac{(2n-3)!}{(n-2)!}t^{n-1}.$$
The constant $c_{n,J}^K$ is
\begin{equation}
\label{eq:almost type d}
p_{s_n}(w_K)\frac{p_{v_J}(w_K)}{w_{v_K}(w_K)}=p_{s_n}(w_K) \frac{\frac{(2n-3)!}{(n-2)!}t^{n-1}}{\frac{(2n-3)!}{(n-2)!}(n-1) t^n}=p_{s_n}(w_K)\frac{1}{(n-1)t}.
\end{equation}
The polynomial $p_{s_n}(w_K)$ is computed in the even and odd cases. In both cases, if $m$ is less than $n-1$ and there is a box corresponding to $s_n$ in row $x_m$, then 
$$x_{n-1}x_{n-2}x_{n-3}\cdots x_{m+1}s_ms_{m+1}\cdots s_{n-3}s_{n-2}(\alpha_n)$$
$$=x_{n-1}x_{n-2}x_{n-3}\cdots x_{m+1}(\alpha_m+\alpha_{m+1}+\cdots \alpha_{n-3}+\alpha_{n-2} + \alpha_n)$$
$$=x_{n-1}x_{n-2}x_{n-3}\cdots x_{m+2}(\alpha_m+2\alpha_{m+1}+\cdots + 2\alpha_{n-3}+2\alpha_{n-2} +\alpha_{n-1}+ \alpha_n).$$
The root $\alpha_m+2\alpha_{m+1}+\cdots + 2\alpha_{n-3}+2\alpha_{n-2} +\alpha_{n-1}+ \alpha_n$ is invariant under the action of $s_i$ for $i>m+1$.  Thus the root corresponding to that box is $$\alpha_m+2\alpha_{m+1}+\cdots + 2\alpha_{n-3}+2\alpha_{n-2} +\alpha_{n-1}+ \alpha_n$$ and the entry in $\lambda_{p(w_K)}$ is $2(n-m)-1$.\\
\\
If $n$ is odd then row $x_{n-1}$ does not contain the reflection $s_n$ and  
\begin{align*}
p_{s_n}(w_K)& = \left[  \sum \limits_{\substack{1 \leq m \leq n-1\\ m \text{ odd}}} 2(n-m)-1 \right] \cdot t\\
& = \left[  \sum \limits_{l=1}^{\frac{n}{2}-1} 2n-4l-1 \right] \cdot t\\
&= \frac{n^2-n}{2} \cdot t.
\end{align*}
\noindent
If $n$ is even then the row $x_{n-1}$ contains only the reflection $s_n$ and that box corresponds to the root $\alpha_n$.  Since $\pi(\alpha_n)=t$ we have
\begin{align*}
p_{s_n}(w_K)&= \left[1+ \sum \limits_{\substack{1\leq m \leq n-3\\ m \text{ odd}}} 2(n-m)-1 \right] \cdot t\\
&= \left[ 1+\sum \limits_{l=1}^{\frac{n}{2}-1} 2n-2(2l-1)-1 \right] \cdot t\\
&=  \left[ 1+\sum \limits_{l=1}^{\frac{n}{2}-1} 2n-4l+1 \right]\cdot t\\
&= \left[  1+ (\frac{n}{2}-1)(2n+1) - 4\frac{(\frac{n}{2}-1)(\frac{n}{2})}{2} \right] \cdot t\\
&=\frac{n^2-n}{2}\cdot t
\end{align*}
\noindent
Thus for any $n>3$ regardless of parity, Equation~\eqref{eq:almost type d} becomes
$$c_{n,J}^K=\frac{n^2-n}{2} t \cdot \frac{1}{(n-1)t}=\frac{n}{2}.$$
   \end{proof}
\subsection{Type $E$}
\begin{figure}[h]
\begin{tikzpicture}
\draw[] (0,0) circle [radius=0.08];
\node [below] at (0,0)  {$s_1$};
\draw[] (1,0) circle [radius=0.08];
\node [below] at (1,0)  {$s_3$};
\draw (0.08,0) -- (.92,0);
\draw (1.08,0) -- (1.92,0);
\draw[] (2,0) circle [radius=0.08];
\node [below] at (2,0)  {$s_4$};
\draw[] (3,0) circle [radius=0.08];
\node [below] at (3,0)  {$s_{n-1}$};
\draw [dashed] (2.08,0) -- (2.92,0);
\draw[] (4,0) circle [radius=0.08];
\node [below] at (4,0)  {$s_{n}$};
\draw [] (3.08,0) -- (3.92,0);
\draw[] (2,1) circle [radius=0.08];
\node [above] at (2,1)  {$s_{2}$};
\draw [] (2,.08) -- (2,.92);
\node at (-1,0){$E_n$:};
\end{tikzpicture}
\end{figure}
\begin{Proposition}
Theorem \ref{thm:by type} holds when $K$ is a connected type $E$ root system.
\end{Proposition}
\begin{proof}
The Giambelli formula for the Peterson varieties in the exceptional types was calculated using Sage code.  While the calculations for types $F_4$ and $G_2$ can easily be reproduced by hand, the $E$ series computations heavily relied on computers.  As such the type $E$ computations will be given with the accompanying code to reproduce the results.  Unlike for the infinite series Lie types, if $K$ is a root system of type $E_n$, the word $w_K$ does not give rise to a nice diagram.  Instead we present that information as a table in Figure \ref{fig:e lists}.\\
\\
The first list $L_{w_K}$ gives the ordered simple reflections $s_{j_i}$ such that $w_K=s_{j_1}s_{j_2}\cdots s_{j_\ell}$ is the reduced word for $w_K$ given by the algebraic combinatorics platform Sage.  The second list $L_{p(w_K)}$  is created using a Sage program.  For each simple reflection $s_{j_i}$ of $w_K$ we record $\frac{1}{t}\cdot \pi_1(\mathbf{r}(\mathbf{i},w_K))$.  The code for this program is available at {http://arxiv.org/abs/1311.2678}.\\

\begin{figure}
\caption{The lists $L_{w_K}$ and $L_{p(w_K)}$ for $K=E_6,E_7,$ and $E_8$. The bold simple reflection in $L_{w_{E_7}}$ and $L_{w_{E_8}}$ is the last occurance of the reflections $s_7$ and $s_8$ respectively. } \label{fig:e lists}
\begin{center} \scalebox{.8}{ \setcounter{MaxMatrixCols}{40}
$\begin{matrix}
L_{w_{E_6}} = & s_1,& s_3,& s_4,& s_5,& s_6,& s_2,& s_4,& s_5,& s_3,& s_4,& s_1,& s_3,& s_2,& s_4,& s_5,& s_6,& s_2,& s_4,\\
& s_5,& s_3,& s_4,& s_1,& s_3,& s_2,& s_4,& s_5,& s_3,& s_4,& s_1,& s_3,& s_2,& s_4,& s_1,& s_3,& s_2,& s_1 \\
L_{p(w_{E_6})} = & 1,& 2,& 3,& 4,& 5,& 4,& 5,& 6,& 6,& 7,& 7,& 8,& 8,& 9,& 10,& 11,& 1,& 2,\\
& 3,& 3,& 4, & 4,& 5,& 5,& 6,& 7,& 1,& 2,& 2,& 3,& 3,& 4,& 1,& 2,& 1,& 1\\
\\
%\end{matrix}$
%\vspace{8mm}
%$\begin{matrix}
L_{w_{E_7}} = & s_7, &s_6, &s_5, &s_4, &s_3, &s_2, &s_4, &s_5, &s_6, &s_7, &s_1, &s_3, &s_4, &s_5,  &s_6, &s_2,   &s_4,   & \\ 
                         &s_5,&s_3,   &s_4, &s_1, &s_3, &s_2, &s_4, &s_5, &s_6, &\mathbf{s_7} ,&s_1, &s_3, &s_4, &s_5, &s_6, &s_2, &s_4, & \\ 
                        &s_5, &s_3, &s_4, &s_1, &s_3, &s_2, &s_4, &s_5, &s_6, &s_2, &s_4, &s_5, &s_3, &s_4, &s_1, &s_3, &s_2, &\\ 
                        &s_4, &s_5, &s_3, &s_4, &s_1, &s_3,&s_2, &s_4, &s_1, &s_3, &s_2, &s_1& \\
L_{p(w_{E_7})} = &1, &2, &3, &4, &5, &5, &6, &7, &8, &9, &6, &7, &8, &9, &10, &9, &10,&\\ &11, 
&11, &12, &12, &13, &13, &14, &15, &16, &17, &1, &2, &3, &4, &5, &4, &5,&\\ &6, &6, &7, &7, &8, &8, &9, &10, &11, &1, &2, &3, &3, &4, &4, &5, &5,&\\ &6, &7, &1, 
&2, &2, &3, &3, &4, &1, &2, &1, &1&\\
\\
%\end{matrix}$
%
%\vspace{8mm}
%\setcounter{MaxMatrixCols}{40}
%$\begin{matrix}
L_{w_{E_8}}= 
& s_8,& s_7,& s_6,& s_5,& s_4,& s_3,& s_2,& s_4,& s_5,& s_6,& s_7,& s_1,& s_3,& s_4,& s_5,& s_6,& s_2, \\
&s_4,& s_5,& s_3,& s_4,& s_1,& s_3,& s_2,& s_4,& s_5,& s_6,& s_7,& s_8,& s_7,& s_6,& s_5,& s_4, &s_3,\\
&s_2,& s_4,& s_5,& s_6,& s_7,& s_1,& s_3,& s_4,& s_5,& s_6,& s_2,& s_4,& s_5,& s_3,& s_4,& s_1,& s_3,\\
& s_2,& s_4,& s_5,& s_6,& s_7,& \mathbf{s_8},& s_7,& s_6,& s_5,& s_4,& s_3,& s_2,& s_4,& s_5,& s_6,& s_7,& s_1,\\
& s_3,& s_4,& s_5,& s_6,& s_2,& s_4,& s_5,& s_3,& s_4,& s_1,& s_3,& s_2,& s_4,& s_5,& s_6,& s_7,& s_1, \\
&s_3,& s_4,& s_5,& s_6, &s_2, &s_4, &s_5, &s_3, &s_4, &s_1, &s_3, &s_2, &s_4, &s_5, &s_6, &s_2, &s_4, \\
&s_5, &s_3, &s_4, &s_1, &s_3, &s_2, &s_4, &s_5, &s_3, &s_4, &s_1, &s_3, &s_2, &s_4, &s_1, &s_3, &s_2, \\
&s_1 &&&&&&&&&&&&&&&&\\
L_{p(w_{E_8})}= 
& 1, &2,& 3,& 4,& 5,& 6,& 6,& 7,& 8,& 9,& 10,& 7,& 8,& 9,& 10,& 11,& 10, \\
&11,& 12,& 12,& 13,& 13,& 14,& 14,& 15,& 16,& 17,& 18,& 29,& 11, &12, &13, &14,& 15,\\
& 15,& 16,& 17,& 18,& 19,& 16,& 17,& 18,& 19,& 20,& 19,& 20,& 21,& 21,& 22,& 22,& 23,\\
& 23, &24,& 25,& 26,& 27,& 28,& 1,& 2,& 3,& 4,& 5,& 5,& 6,& 7,& 8,& 9,& 6,\\
& 7,& 8,& 9,& 10,& 9,& 10,& 11,& 11,& 12,& 12,& 13,& 13,& 14,& 15,& 16,& 17,& 1,\\
& 2,& 3,& 4,& 5,& 4,& 5,& 6,& 6,& 7,& 7,& 8,& 8,& 9,& 10,& 11, &1, &2,\\
& 3,& 3,& 4,& 4,& 5,& 5,& 6,& 7,& 1,& 2,& 2,& 3,& 3,& 4, &1, &2, &1, \\
&1 &&&&&&&&&&&&&&&&
\end{matrix}$}
\end{center}
\end{figure}
\noindent
The only reduced words of $v_K$ are $$v_K=s_1s_2s_3s_4\cdots s_n=s_2s_1s_3s_4\cdots s_n=s_1s_3s_2s_4\cdots s_n.$$
Python code can find all sublists of $L_{w_K}$ that are equal to one of the three corresponding lists.  These sublists are the $v_K$-excitations $\mu$ of $w_K$.  For each excitation $\mu$, $M(\mu)$ is the product of the entries in $L_{p(w_K)}$ in the same positions as those in the sublist $\mu$.  We then sum over all such $\mu$ to get
$$p_{v_K}(w_K)=\left[ \sum \limits_{\substack{\mu \text{ a $v_K$-excitation}} \\ \text{ of  }L_{w_K} }M(\mu) \right] \cdot t^n.$$
We also evaluate $p_{s_i}(w_K)$ by summing all of the entries in $L_{p(w_K)}$ corresponding to entries which are $s_i$ in $L_{w_K}$. 
This gives the following data for the $E$ series:
$$\begin{matrix}
& \text{Type }E_6& \text{Type }E_7& \text{Type }E_8\\
p_{v_K}(w_K) & 887040t^6 &661620960t^7 &11179629901440t^8\\
p_{s_1}(w_K)& 16t &34t &92t\\
p_{s_2}(w_K)&22t &49t &136t\\
 p_{s_3}(w_K)& 30t  &66t & 182t \\
p_{s_4}(w_K)& 42t &  96t & 270t\\
p_{s_5}(w_K)& 30t  &75t  & 220t  \\
p_{s_6}(w_K)&16t & 52t &168t\\
 p_{s_7}(w_K)&&27t&114t\\
 p_{s_8}(w_K)&&&172t
\end{matrix}$$
This table along with Lemma \ref{lemma one product} gives us that the $E$ series Peterson varieties have the following Giambelli's formula.
$$\begin{matrix}
E_6: &  C \cdot p_{v_K}(w_K)=\prod \limits_{1 \leq i \leq 6} p_{s_i}(w_K) & \text{where } &C=240=\frac{6!}{3}\\
E_7: &  C \cdot p_{v_K}(w_K)=\prod \limits_{1 \leq i \leq 7} p_{s_i}(w_K) & \text{where } &C=680=\frac{7!}{3}\\
E_8: &  C \cdot p_{v_K}(w_K)=\prod \limits_{1 \leq i \leq 8} p_{s_i}(w_K) & \text{where } &C=13440=\frac{8!}{3}\\
\end{matrix}$$
   \end{proof}

\subsection{Type $F_4$}
\begin{figure}[h]
\begin{tikzpicture}
\draw[] (0,0) circle [radius=0.08];
\node [below] at (0,0)  {$s_1$};
\draw[] (1,0) circle [radius=0.08];
\node [below] at (1,0)  {$s_2$};
\draw (0.08,0) -- (.92,0);
\draw[] (2,0) circle [radius=0.08];
\node [below] at (2,0)  {$s_3$};
\draw[] (3,0) circle [radius=0.08];
\node [below] at (3,0)  {$s_{4}$};
\draw [] (2.08,0) -- (2.92,0);
\draw (1.05,.05) -- (1.95,.05);
\draw (1.05,-.05) -- (1.95,-.05);
\draw (1.45,.1) -- (1.55,0) -- (1.45,-.1);
\node at (-1,0){$F_4$:};
\end{tikzpicture}
\end{figure}
\begin{Proposition}
Theorem \ref{thm:by type} holds when $K$ is a connected type $F_4$ root system.
\end{Proposition}
\begin{proof}
If $K$ is a connected root system of type $F_4$ then $J=K \setminus \lbrace s_4 \rbrace$ is a root subsystem of type $B_3$ and therefore
$$c^{K}_{4,J} \cdot p_K=p_{s_4}p_{v_J}= \frac{1}{3!} \prod \limits_{i=1}^4 p_i.$$
Evaluating at $w_K=s_4s_3s_2s_3s_1s_2s_3s_4s_3s_2s_3s_1s_2s_3s_4s_3s_2s_3s_1s_2s_3s_1s_2s_1$ gives $$ \begin{array}{ccc}
&p_K(w_K)=18480t^4=2^4\cdot 3\cdot 5 \cdot 7 \cdot 11 \cdot t^4 & \\
p_1(w_K)=22t &&
p_2(w_K)=42t\\
p_3(w_K)= 30t&&
p_4(w_K)=16t
\end{array}$$
Since $$c^{K}_{4,J}\cdot p_K(w_K)=\frac{1}{3!}\prod \limits_{i=1}^4 p_i(w_K),$$
we can solve for $c^{K}_{4,J}$ to see that $c^{K}_{4,J}=4$.  Thus 
$4! \cdot p_K= (|K|)!\cdot p_K=\prod \limits_{i=1}^4 p_i.$
   \end{proof}

\subsection{Type $G_2$}
\begin{figure}[h]
\begin{tikzpicture}
\draw[] (0,0) circle [radius=0.08];
\node [below] at (0,0)  {$s_1$};
\draw[] (1,0) circle [radius=0.08];
\node [below] at (1,0)  {$s_2$};
\draw (0.08,0) -- (.92,0);
\draw (0.05,.05) -- (0.95,.05);
\draw (.05,-.05) -- (.95,-.05);
\draw (.55,.1) -- (.45,0) -- (.55,-.1);
\node at (-1,0){$G_2$:};
\end{tikzpicture}
\end{figure}
\begin{Proposition}
Theorem \ref{thm:by type} holds when $K$ is a connected type $G_2$ root system.
\end{Proposition}
\begin{proof}
Since in type $G_2$ the ring $H_{S^1}^*(Pet_{G_2})$ has only four Peterson Schubert classes, we give the basis explicitly evaluated at $e,s_1,s_2,$ and $s_1s_2s_1s_2s_1s_2$:
$$\begin{array}{cccc}
p_\emptyset & p_{\lbrace 1 \rbrace} & p_{\lbrace 2 \rbrace} & p_{\lbrace 1,2 \rbrace} \\
\begin{pmatrix}
1\\
1\\
1\\
1\\
\end{pmatrix} 
&
\begin{pmatrix}
0\\
t\\
0\\
6t\\
\end{pmatrix} 
&
\begin{pmatrix}
0\\
0\\
t\\
10t\\
\end{pmatrix} 
&
\begin{pmatrix}
0\\
0\\
0\\
30t^2\\
\end{pmatrix} \\
\end{array}$$
From this basis we see that $2\cdot p_K = \prod \limits_{i\in K} p_{s_i}$ and of course $|K|!=2$. 
   \end{proof}

\section{Acknowledgments} 
Thank you to Julianna Tymoczko for her help and guidance at all stages of this project.  Thank you to Dave Anderson, Tom Braden,  David Cox,  Allen Knutson, Aba Mbirika, Jessica Sidman, and Alex Yong  for helpful comments and conversations.  I am grateful to Erik Insko for getting me started with Sage and to Hans Johnston and Jeff Hatley for the computer access and programming help.  Lastly, thank you to the anonymous referees for their detailed comments on how to improve this manuscript. {This work was partially supported by NSF grant DMS-1248171.}

\end{document}